\documentclass[leqno]{amsart}
\usepackage{amsfonts,amssymb,amsmath,amsgen,amsthm}
\usepackage{hyperref}
\usepackage{color}

\theoremstyle{plain}
\newtheorem{theorem}{Theorem}[section]

\newtheorem{lemma}[theorem]{Lemma}

\newtheorem{proposition}[theorem]{Proposition}
\newtheorem{hyp}[theorem]{Assumption}
\theoremstyle{remark}
\newtheorem{remark}[theorem]{Remark}
\newtheorem*{notation}{Notation}

\def\C{{\mathbf C}}% complex numbers
\def\R{{\mathbf R}}% real numbers
\def\N{{\mathbf N}}% nonnegative integers
\def\O{\mathcal O}

\def\({\left(}
\def\){\right)}
\def\<{\left\langle}
\def\>{\right\rangle}
\def\le{\leqslant}
\def\ge{\geqslant}

\def\Tend#1#2{\mathop{\longrightarrow}\limits_{#1\rightarrow#2}}

\def\si{{\sigma}}

\def\ta{{\tt a}^\hbar}
\def\tR{{\tt R}^\hbar}
\def\tw{{\tt w}^\hbar}

\def\C{{\mathbb C}}
\def\R{{\mathbb R}}
\def\N{{\mathbb N}}

\def\d{{\partial}}

\def\uu{{\bf u}}

\def\ee{{\rm e}}

\def\eps{\varepsilon}
\def\op_#1{\mathrel{\mathop{{\rm op}_{#1}}}}

\def\d{{\partial}}

\def\build#1_#2^#3{\mathrel{
\mathop{\kern 0pt#1}\limits_{#2}^{#3}}}

\def\td_#1,#2{\mathrel{
\mathop{\build\longrightarrow_{#1\rightarrow #2}^{}}}}

\def\lim_#1,#2{\mathrel{
\mathop{\build{\rm lim}_{#1\rightarrow#2}^{}}}}

\def\eps{\varepsilon}

\def\op{{\rm op}}
\def\ol{\overline}

\DeclareMathOperator{\RE}{Re}
\DeclareMathOperator{\IM}{Im}

\numberwithin{equation}{section}

\begin{document}

\title[Nonlinear Landau-Zener formula]{A nonlinear Landau-Zener formula}
\author[R. Carles]{R\'emi Carles}
\begin{abstract}
  We consider a system of two coupled ordinary differential equations
  which appears as an envelope equation in Bose--Einstein
  Condensation. This system can be viewed as a nonlinear extension of
  the celebrated model introduced by Landau and Zener. We show how the
  nonlinear system may appear from different physical models. We focus
  our attention on the large 
  time behavior of the solution. We show the existence of a nonlinear
  scattering operator, which is reminiscent of long range scattering
  for the nonlinear Schr\"odinger equation, and which can be compared
  with its linear counterpart. 
\end{abstract}
\thanks{This work was supported by the French
ANR projects
  R.A.S. (ANR-08-JCJC-0124-01) and SchEq
  (ANR-12-JS01-0005-01).}
\address{CNRS \& Univ. Montpellier 2\\ Math\'ematiques, CC~051\\ 34095
  Montpellier\\ France}
\email{Remi.Carles@math.cnrs.fr}
\author[C. Fermanian]{Clotilde~Fermanian-Kammerer}
\address{LAMA, UMR CNRS 8050,
Universit\'e Paris EST\\
61, avenue du G\'en\'eral de Gaulle\\
94010 Cr\'eteil Cedex\\ France}
\email{Clotilde.Fermanian@u-pec.fr}
\maketitle

\section{Introduction}
\label{sec:intro}

\subsection{Physical motivation}
\label{sec:motiv}

In the 30's, questions about adiabaticity have  begun to be studied
and major contributions have been brought by Landau and Zener,
independently (see~\cite{La} and~\cite{Ze}). The system studied by
Zener reads 
\begin{equation}\label{eq:LZ}
-i \partial_s u=
V(s,z) u ; \quad u(0)=u_0,
 \end{equation}
where $u=\begin{pmatrix}u_1\\u_2\end{pmatrix}\in \C^2$, and the potential $V$ is given by  
$$V( s , z ) = 
\begin{pmatrix}
   s & z \\
 z & - s
\end{pmatrix}.
$$
This system presents an eigenvalue crossing: the eigenvalues
of the matrix $V(s,z)$ are $\sqrt{s^2+z^2}$ and $-\sqrt{s^2+z^2}$, and
they cross when $s=0$ and $z=0$.  One can easily check that the
eigenspace associated with 
$\sqrt{s^2+z^2}$ is asymptotic to $\R\begin{pmatrix}1\\
  0\end{pmatrix}$ when $s\rightarrow +\infty$, and to
$\R\begin{pmatrix}0\\ 1\end{pmatrix}$ when $s\rightarrow
-\infty$. Therefore, a  solution $u(s,z)$ asymptotic to  $\begin{pmatrix}0
  \\\alpha(s,z)\end{pmatrix}$ when $s\rightarrow -\infty$ is polarized
along the $+$ mode. The adiabatic issue consists in asking whether
such a solution is still polarized along the $+$ mode when
$s\rightarrow +\infty$. The adiabatic character of this solution is
characterized by the ratio   $|u_1(s,z)|^2/|\alpha(s,z)|^2$.
It has been  proved in the 30's that for~\eqref{eq:LZ}, this 
ratio is described by the Landau-Zener  transition coefficient~${\rm
  e}^{-\pi z ^2}$ which characterizes the transfer between modes (see
Section~\ref{sec:scat-intro} below for details).  More generally, for
an eigenvalue gap of the form~$\sqrt{a^2s^2+b^2z^2}+\O(s^2)$ near
$s=0$, the transition coefficient was expected to be ${\rm
  e}^{-\pi z^2b^2/ a}$. 

$ $

Mathematical  proofs of the Landau-Zener formula have been given in
the 90's by Hagedorn in~\cite{Hagedorn91} for small gaps and, then  by
Joye in~\cite{Joye94}.  Later,
in~\cite{CdV1}, \cite{CdV2}, Colin de Verdi\`ere has performed a
classification of pseudodifferential operators with symbols presenting
crossings of two eigenvalues.  Under some genericity assumptions that
we will not detail here, he derives simple  model systems  thanks to
a change of coordinates in the phase space (a canonical transform) and
the conjugation by a Fourier Integral Operator  (including a Gauge
transform).   In the case of a system of evolution equations, the
reduced systems are closely related with~(\ref{eq:LZ})
(see~(\ref{eq:LZ2}) below); for this reason, the system~(\ref{eq:LZ})
appears as a toy-model for understanding  eigenvalue crossings.  

$ $

Recently, a nonlinear version of
this system has been introduced in \cite{BiQi00}, of the form
\begin{equation}\label{eq:modelphys1}
  i\frac{\partial}{\partial t}
  \begin{pmatrix}
    u_1 \\
u_2
  \end{pmatrix}
=
H(\gamma)
  \begin{pmatrix}
    u_1 \\
u_2
  \end{pmatrix},
\end{equation}
with the Hamiltonian given by
\begin{equation}\label{eq:modelphys2}
  H(\gamma) = 
  \begin{pmatrix}
    \gamma( t) + \delta\(|u_2|^2-|u_1|^2\) & z\\
z & -\gamma (t)  -\delta \(|u_2|^2-|u_1|^2\) 
  \end{pmatrix},
\end{equation}
where, according to the terminology in \cite{BiQi00},
$\gamma(t)=\gamma_1 t$ denotes the level separation, $z$ is the
coupling 
constant between the two levels, and $\delta$ is a parameter
describing the nonlinear interaction. This nonlinear two-level model can be used
to understand Landau-Zener tunneling of a Bose-Einstein condensate
between Bloch bands in an optical lattice: this has been observed in 
e.g. \cite{PRA02}. In Section~\ref{sec:deriv}, we present a derivation
of this model from the nonlinear Schr\"odinger equation in a rotating
frame and  a periodic potential. Physical properties of the above
system have been investigated in
e.g. \cite{PRL03,KhRu05,Kh10,PRA02b}. We also show how this model can
arise under the influence of a double-well potential, as in
\cite{Kh10} (see Section~\ref{sec:deriv} for a rapid presentation, and
the appendix for a rigorous proof). 
The goal of this paper is to
recast this model in a mathematical framework, and to study some of
its properties, in particular as far as the large time regime is
concerned.  

\subsection{Mathematical setting}
\label{sec:setting}

The classification of crossings
performed in \cite{CdV1} and~\cite{CdV2} in the one hand, and in  \cite{FG02}
on the other hand,  produces  model problems of the form  
 \begin{equation}\label{eq:LZ2}
-i \partial_s u=
\begin{pmatrix}
   s & G \\
 G^* & - s
\end{pmatrix} u;\quad u(0)=u_0\in{\mathcal H}^2,
 \end{equation}
 where ${\mathcal H} $ is a Hilbert space and $G$ an operator on
 ${\mathcal H}$.  Equation~\eqref{eq:LZ}  corresponds to ${\mathcal
   H}=\C$ and $G=z\in\R$, as in~\cite{La} and~\cite{Ze} (see
 also~\cite{FG02} and~\cite{CdV1}).  Other choices of the pair
 $({\mathcal H},G)$ are relevant:  
 \begin{enumerate}
 \item   In~\cite{CdV2} and~\cite{Ha94}, $G=z_1+iz_2$ and ${\mathcal
     H}=\C$. 
 \item  In~\cite{CdV2} and~\cite{Fe04}, $G=\partial_z-z$ and
   ${\mathcal H}=L^2(\R)$.
 \item In~\cite{Ha94}, \cite{FL08} and~\cite{Fe03},  $G =q(z)$ is a
   quaternion, $z\in\R^4$ and ${\mathcal H}=\C^2$.
 \item In~\cite{Fe03} and~\cite{FG02}, $G$ is a semiclassical pseudodifferential
   operator on the space ${\mathcal H}=L^2(\R^k)$, $k=1,2,3\dots$
 \end{enumerate}

  For this reason, we will focus on the following abstract problem
   \begin{equation}\label{eq:NLLZ2}
-i \partial_s u=
\begin{pmatrix}
   s & G \\
 G^* & - s
\end{pmatrix} u +\delta  F(u)u;\quad u(0)=u_0\in{\mathcal H}^2,
 \end{equation}
  which is the nonlinear counterpart of~\eqref{eq:LZ2} and contains the systems coming from physics such as~(\ref{eq:modelphys1})--\eqref{eq:modelphys2}. 
The nonlinearity $F:{\mathcal H}^2\to {\mathcal L}({\mathcal H}^2)$ is of the form 
$$  F(u) =  {\rm diag}(F_1(u),F_2(u)),$$
where for $(A,B)\in\R^2$, the operator ${\rm diag}(A,B)$ acts on
${\mathcal H}^2$ as  
$${\rm diag}(A,B) (u_1,u_2)=(Au_1,Bu_2),$$ 
and where the functions $F_j:\C^2\to\R$ satisfy 
\begin{equation}
  \label{def:f}
  F_j(u)=
    f_j(\|u_1\|^2_{\mathcal H},\|u_2\|^2_{\mathcal H}),
\end{equation}
with $f_j\in{\mathcal C}^2(\R^2;\R)$ and $\nabla^2f_j$
bounded. By the change of variable
$s=t\sqrt{\gamma_1}$, we are left with~(\ref{eq:NLLZ2}) with
$G=z/\sqrt{\gamma_1}$ and $F_j(u)=(-1)^j\delta\gamma_1^{-1/2}
(|u_2|^2-|u_1|^2)$ for $j\in\{1,2\}$. 

\smallbreak

  We make the following assumption on the operator $G$: for any  $\chi\in{\mathcal C}^0(\R)$,
  \begin{equation}\label{commutation}
  \chi(GG^*)G=G\chi(G^*G)\;\;{\rm and}\;\;\chi(G^*G)G^*=G^*\chi(GG^*),
  \end{equation}
  where the operators $\chi(GG^*)$ and $\chi(G^*G)$ are defined by functional calculus.
  This assumption is satisfied in the  first three examples
  above. The situation is more complicated in the fourth one where it only holds at leading order in the semiclassical parameter.

We also assume that the domains ${\mathcal
  D}(GG^*)$ and ${\mathcal D}(G^*G)$  are dense subsets of
${\mathcal H}$.  

  \begin{lemma}\label{lem:existence}
    Let $u_0\in {\mathcal H}^2$. Under the above assumptions,
    \eqref{eq:NLLZ2} has a unique, global solution $u \in
    C(\R;{\mathcal H}^2)$. It satisfies the following
    conservation law
\begin{equation*}
  \frac{d}{ds}\(\|u_1(s)\|_{\mathcal H}^2 + \|u_2(s)\|_{\mathcal H}^2\)=0. 
\end{equation*}
  \end{lemma}
  \begin{proof}
    Denote by $U(s_2,s_1)$ the linear operator which maps $u^{\rm lin}(s_1)$ to
    $u^{\rm lin}(s_2)$, where $u^{\rm lin}$ solves the linear equation
    \begin{equation*}
      -i \partial_s u^{\rm lin}=
\begin{pmatrix}
   s & G \\
 G^* & - s
\end{pmatrix} u^{\rm lin} .
    \end{equation*}
It satisfies $U(s,s)={\rm Id}$, $U(s,\tau)U(\tau,\si)=U(s,\si)$, and
is unitary on ${\mathcal H}^2$. 
By Duhamel's principle, \eqref{eq:NLLZ2} becomes
\begin{equation*}
  u(s) = U(s,0)u_0 + i \delta \int_0^s U(s,\si) \(F(u)u\)(\sigma)d\sigma.
\end{equation*}
Local existence follows by a standard fixed point argument
(Cauchy-Lipschitz), with a local existence time which depends only on
$\|u_0\|_{{\mathcal H}^2}$. We then note the conservation
law announced in the lemma, which follows from the fact that the
functions $F_j$ are real-valued. This implies global existence, and the
lemma. 
  \end{proof}
Throughout this paper, we consider normalized initial data
so that we have 
\begin{equation}\label{eq:conserv}
  \|u_1(s)\|_{\mathcal H}^2 + \|u_2(s)\|_{\mathcal H}^2=1,\quad
  \forall s\in \R.
\end{equation}
 We prove a scattering result for initial data
   which are in the range of ${\bf 1}_{V(0)^2\le R}$ for some $R>0$.  
   More precisely, we introduce  a cut-off operator depending on $G$:
   let $\theta$ be  
 a smooth cut-off function , $\theta\in{\mathcal C}_0^\infty(\R)$ with
 $0\le \theta\le 1$,  $\theta(u)=0$ for $|u|>1$ and $\theta(u)=1$ for
 $|u|<1/2$. Then,  for $R>0$ we set 
  $$
  \Theta_R={\rm diag}\left( \theta\left({GG^*\over R^2}\right),
   \theta\left({G^*G\over R^2}\right) \right).
  $$
  Because of~(\ref{commutation}), the operator $\Theta_R$ commutes with $V(s)$ for all $s\in \R$ and
  $\Theta_RV(s)$ is a bounded operator on ${\mathcal H}\times{\mathcal
    H}$ with norm $\sqrt{s^2+R^2}$.  Besides,  a simple computation  shows that $u_R(s)=\Theta_R u(s)$ satisfies 
$\frac{1}{i} \partial_s u_R=V(s) u_R + F(u) u_R$ with  $u_R(0)=\Theta_R
u_0=u_0$. Therefore,  we have the following result.

  \begin{lemma}\label{lem:localisation}
  Suppose $u_0=\Theta_Ru_0$ for some $R>0$, then for all $s\in\R$, the
  solution of~\eqref{eq:NLLZ2} satisfies $u(s)=\Theta_Ru(s)$. 
  \end{lemma} 
 Typically, in the physical examples presented in
 Section~\ref{sec:deriv}, the assumption $u_0=\Theta_Ru_0$ consists in
 saying that some physical parameter (whose value is fixed in
 practise) belongs to a bounded set.
 
 \begin{notation}
   For two  real-valued functions $a(s)$ and $b(s)$, we write 
   \begin{equation*}
     a(s) = \O_R\(b(s)\),\quad s\in I,
   \end{equation*}
whenever there exists $C_R$ depending only on $R$ such that
\begin{equation*}
  |a(s)|\le C_R|b(s)|, \quad s\in I.
\end{equation*}
Similarly, for $u(s)\in \mathcal H$, we write $u(s)= \O_R\(b(s)\)$
if $\|u(s)\|_{\mathcal H} = \O_R\(b(s)\)$.
 \end{notation}

\subsection{Scattering}
\label{sec:scat-intro}

We introduce the phase  function $\varphi$ given by 
$$
\varphi(s,\lambda )={s^2\over 2} + {\lambda\over 2} {\ln} |s|.
$$
We can describe the large time asymptotics of $u$. 
We write all the large time results in the case where
  ${\mathcal H}$ 
  is a vector space of finite dimension, which implies that the unit
  ball of ${\mathcal H}$ is compact.

\begin{theorem}\label{theo:scat}

 \noindent Assume that $u_0=\Theta_R u_0$ for some $R>0$. Then, there
 exist unique pairs $\alpha=(\alpha_1,\alpha_2)\in{\mathcal H}^2$ and 
 $\omega=(\omega_1,\omega_2)\in{\mathcal H}^2$, 
   such that:\\
$1.$ As $s$ goes to $-\infty$,
    \begin{align*}
     u_1(s) & =  {\rm e} ^{i\delta F_1(\alpha)s+ i\varphi(s,GG^*)}
    \alpha_{1}+\O_R\(\frac{1}{|s|}\),\\
    u_2(s) & =  {\rm e} ^{i\delta F_2(\alpha)s-i\varphi(s,G^*G)}
    \alpha_{2}+\O_R\(\frac{1}{|s|}\).
    \end{align*}
$2.$ As $s$ goes to $+\infty$,
   \begin{align*}
    u_1(s) & =  {\rm e} ^{i\delta F_1(\omega)s+i\varphi(s,GG^*)}
    \omega_{1}+\O_R\(\frac{1}{s}\),\\
    u_2(s) & =  {\rm e} ^{i\delta F_2(\omega)s-i\varphi(s,G^*G)}
    \omega_{2}+\O_R\(\frac{1}{s}\).
    \end{align*}
\end{theorem}
The above result is reminiscent of long range scattering in
nonlinear Schr\"odinger equations, as described first in
\cite{Ozawa91}: nonlinear effects are present both in the fact that
the amplitude of the functions $u_j$ undergoes a nonlinear influence
(this is what happens in nonlinear scattering in general),
and in the fact that oscillations are different from those of the
linear case (a typical feature of long range scattering), since the
linear phase $\varphi$ is not enough to describe large time
oscillations. An important difference though is that \eqref{eq:NLLZ2}
is not a dispersive equation, as can be seen from \eqref{eq:conserv}. 

\begin{remark}\label{rem:scatlin} When $\delta=0$ and $G$ is scalar
  ($G=z$), this result goes back to the 30's with the proofs~of Landau
  and Zener~\cite{La} and~\cite{Ze}. The original proof is based on the
  use of special functions; more recently, the proof of~\cite{FG00}
  relies on the analysis of oscillatory integrals. When $G$ is
  operator-valued, the theorem is proved in the 
  linear case ($\delta=0$) in~\cite[Proposition~7]{FG03}. We
  point out that there is a slight difference with the present
  framework,   due to nonlinear effects. In the linear case, one 
  associates with $u_0$ scattering states such that the asymptotics
  hold true for $\Theta_R u(s)$. Here, the scattering states depend on
  $R$ in a non trivial way.  
\end{remark}

\begin{remark}\label{rem:scatgen}
If the unit ball of ${\mathcal H}$ is not compact,  the  proof  of Section~\ref{sec:scat}  gives that for all $\eps>0$, there exists $s_\eps>0$, 
$\omega^\eps=(\omega_1^\eps,\omega_2^\eps)\in{\mathcal H}^2$ and some constants $(\Omega_+,\Omega_-)\in\R_+^2$ such that 
for all $s>s_\eps$, 
$$
\displaylines{
\left\| u_1(s) -  {\rm e}^{is\delta f_1(\Omega^+,\Omega^-) + i\varphi(s,GG^*)}\omega _1^\eps \right\| _{\mathcal H}\le \eps,\cr
\left\| u_2(s)- {\rm e}^{is\delta f_2(\Omega^+,\Omega^-) -i\varphi(s,G^*G)}\omega
  _2^\eps \right\|_{\mathcal H}\le \eps.\cr}$$
  Moreover, for $j\in\{1,2\}$, $F_j(u(s))$ goes to $f_j(\Omega_+,\Omega_-)$ as $s$ goes to $+\infty$.
\end{remark}

Conversely, wave operators are well-defined, as stated in the
following result.

\begin{proposition}\label{prop:waveop}
 Let $R>0$ and $\omega=(\omega_1,\omega_2)\in{\mathcal
    H}^2$ such that $\omega=\Theta_R \omega$. There exist $\phi\in
  [0,2\pi)$ and  a  
   solution $u \in C(\R;\mathcal H^2)$ to
  \eqref{eq:NLLZ2} such that, as $s\to +\infty$,
  \begin{align*}
& u_1(s) = {\rm e}^{i\delta
     F_1(\omega)s+ i\varphi(s,GG^*)+i\phi } 
    \omega_{1}+\O_R\(\frac{1}{s}\), \\
&u_2 (s) = {\rm e}^{i\delta
      F_2(\omega)s-i\varphi(s,G^*G)-i\phi} \omega_{2}+\O_R\(\frac{1}{s}\).
  \end{align*}
For a fixed $\phi$, such a function $u$ is unique.\\
In addition, if $F_1=F_2$, then we may take $\phi=0$. 
\end{proposition}

\begin{remark} A similar result holds as  $s\rightarrow -\infty$.
\end{remark}
    \begin{remark}
     It is not clear whether $\phi$ can be different from zero. It may
     be seen as a second order influence of the 
     nonlinear term (recall that $\phi=0$ if $F_1=F_2$), the leading
     order term (in $s$) being the oscillatory factor ${\rm e}^{i\delta
      F_j(\omega)s}$. Thus, if $\phi$ is not trivial, it should be
    considered as a function of $\omega$ and $\delta(F_1-F_2)$; the proof of
    Proposition~\ref{prop:waveop} will show why $F_1-F_2$ is involved,
    and not more generally $(F_1,F_2)$. 
   \end{remark}
We can therefore define a scattering operator, which may depend on
$\phi$. Given $\alpha \in {\mathcal H}^2$ and a suitable $\phi$,
Proposition~\ref{prop:waveop} yields a solution $u$ to
\eqref{eq:NLLZ2}, and Theorem~\ref{theo:scat} provides an asymptotic
state $\omega \in {\mathcal H}^2$. The scattering operator maps
$\alpha$ to $\omega$: $\omega = S_\delta^\phi(\alpha)$. 
\smallbreak

Our final result concerning the large time behavior of solutions to
\eqref{eq:NLLZ2} consists in describing the effect of the nonlinearity
in $S_\delta^\phi$ by comparing this operator with its linear counterpart. 
We denote by 
$u^{\rm lin}$ the solution of~\eqref{eq:LZ2}, 
and we associate  linear scattering states with $u_0$ denoted by 
$$\alpha^{\rm lin}=\left(\alpha_{1}^{\rm lin},\alpha_{2}^{\rm
    lin}\right)\;\;{\rm and } \;\;\omega^{\rm
  lin}=\left(\omega_{1}^{\rm lin},\omega_{2}^{\rm lin}\right).$$ 
According to Proposition~7 in~\cite{FG03}, the linear scattering
operator is given by  
    \begin {equation}\label{def:Slin}
    S^{\rm lin}=\begin{pmatrix}a(GG^*) & -\ol b(GG^*)G \\
    b(G^*G)G^* & a(G^*G)\end{pmatrix},
    \end{equation}
    with 
    $$\displaylines{
    a(\lambda)={\rm
      e}^{-\pi\lambda/2},\;\;b(\lambda)=\frac{2i{\rm
        e}^{i\pi/4}}{\lambda \sqrt\pi}2^{-i\lambda/2}{\rm 
e}^{-\pi\lambda/4}\, \Gamma\(1+i{\frac{\lambda}{2}}\) \, \sinh
\({\frac{\pi \lambda}{ 2}}\),\cr 
{\rm and}\;\;a(\lambda)^2+\lambda|b(\lambda)|^2=1.\cr}$$
When $G=z$, the 
 coefficient 
 \begin{equation}\label{LZcoeft}
 T(z)=a(z^2)^2={\rm e}^{-\pi z^2}
 \end{equation}
  is the 
 Landau-Zener transition coefficient which describes the ratio
 ${|\omega^{\rm lin}_1|^2}/ |\alpha^{\rm lin}_1|^2$ of the energy which
 remains on the first component (when $\alpha^{\rm lin}_2=0$). As we shall
 see in the next result, the Landau-Zener transition probability remains
 relevant in the nonlinear regime and for small~$\delta$. 
\begin{theorem}\label{theo:dev}
As  $\delta$ goes to zero, we have the uniform estimate
\begin{equation*}
  S_\delta^\phi = S^{\rm lin} +\O_R(\delta).
\end{equation*}
In particular, at leading order, $S_\delta^\phi$ does not depend on
$\phi$: $\phi=\O_R(\delta)$. \\
If $F_1=F_2=\underline F$, we choose $\phi=0$, and we have the exact formula:
\begin{equation}\label{eq:scatexpl}
  S_\delta (\alpha) ={\rm e}^{i\delta \Lambda^+}S^{\rm lin} \({\rm
  e}^{i\delta\Lambda^-} \alpha\),
\end{equation}
where
\begin{equation*}
 \Lambda^- = \int_{-\infty}^0 \left(\underline F\( u^{\rm
    lin}(\tau)\)-\underline F\(\alpha^{\rm lin}\)\right)d\tau,\quad
 \Lambda^+ = \int_0^{+\infty} \left(\underline F\( u^{\rm
    lin}(\tau)\)-\underline F\(\omega^{\rm lin}\)\right)d\tau. 
\end{equation*}
In particular,
$$\left\|S_{\delta}-S^{\rm lin}-i \delta\, \left(\Lambda^+ +
    \Lambda^-\right)S^{\rm lin}\right\|_{{\mathcal L}({\mathcal
    H}\times{\mathcal H})}=\O_R\( \delta^2\).$$ 
\end{theorem}

As expected, as $\delta \to 0$, $S_\delta^\phi$ behaves at leading order
like the linear scattering operator. Nonlinear effects show up in the
$\O(\delta)$ corrector that can be computed explicitly in the case
$F_1=F_2$
 (the nonlinearity is present in the definition
of $\Lambda^-$ and $\Lambda^+$).

This paper is organized as follows. In Section~\ref{sec:deriv}, we
sketch the derivation of models of the form \eqref{eq:NLLZ2} from cubic
nonlinear Schr\"odinger equations used to describe Bose--Einstein
Condensation. In Section~\ref{sec:prelim}, we  present an algebraic
reduction which we use to recast Theorem~\ref{theo:scat} in terms of another
unknown function (Proposition~\ref{prop:scatv}), and 
we set up some technical
tools needed for the large time study of
\eqref{eq:NLLZ2}. Section~\ref{sec:scat} is dedicated to the proof of
the  
intermediary Proposition~\ref{prop:scatv}.
Proposition~\ref{prop:waveop} is proved in Section ~\ref{sec:waveop}, and
Theorem~\ref{theo:dev} in proved in Section~\ref{sec:dev}. Finally, in 
Appendix~\ref{sec:justif}, we go back to the derivation of \eqref{eq:NLLZ2} from
physical models, and establish rigorously that \eqref{eq:NLLZ2} can be
interpreted as an envelope equation in the semi-classical limit.

\section{Formal derivation of the model}
\label{sec:deriv}

We rapidly describe some cases where \eqref{eq:NLLZ2} appears as an
approximation to describe the motion of a Bose--Einstein condensate. 
\subsection{Condensate in an accelerated optical lattice}
As proposed in \cite{BiQi00}, and considered in e.g. \cite{PRL03},  the
motion of a Bose--Einstein condensate in an accelerated~1D optical
lattice is described by the equation
\begin{equation*}
  i\hbar \frac{\d \psi}{\d t} = \frac{1}{2m}\(-i \hbar \frac{\d}{\d x}- \omega
  t\)^2 \psi +V_0 \cos\(2k_L x\)\psi + \frac{4\pi\hbar^2 a_s}{m}|\psi|^2\psi,
\end{equation*}
where $m$ is the atomic mass, $k_L$ is the optical lattice wave number,
$V_0$ is the strength of the periodic potential depth, $\omega$ is the
inertial force, and $a_s$ is the scattering length. After rescaling,
this equation reads 
\begin{equation}
   \label{eq:nlsbloch}
  i\frac{\d \psi}{\d t} =\frac{1}{2}\(-i \frac{\d}{\d x} -\alpha
  t\)^2\psi+v\cos(x)\psi + \epsilon |\psi|^2\psi,
\end{equation}
with $\epsilon=-1$ or $+1$ according to the chemical element
considered. The approach in~\cite{BiQi00} consists in substituting the
ansatz
\begin{equation}\label{ansatz1}
  \psi(t,x) = a(t) {\rm e}^{ikx} + b(t){\rm e}^{i(k-1)x}
\end{equation}
into \eqref{eq:nlsbloch}, with $k=k_L=1/2$, corresponding to the Brillouin
zone edge (this approximation amounts to considering that only the
ground state and the first excited state are populated, see
\cite{PRL03}). We compute 
\begin{align*}
  i\d_t \psi &= i\dot a {\rm e}^{ikx}+i\dot b {\rm e}^{i(k-1)x},\\
 \(-i\d_x -\alpha t\)^2\psi &= \(k-\alpha t\)^2 a {\rm e}^{ikx} +
\(k-1-\alpha t\)^2  b {\rm e}^{i(k-1)x},\\
\cos(x)\psi &= \frac{1}{2}\( a {\rm e}^{i(k+1)x} +b  {\rm e}^{ikx}+ a{\rm e}^{i(k-1)x} +b
  {\rm e}^{i(k-2)x}\), \\
|\psi|^2\psi&= \(|a|^2+|b|^2\)\(a {\rm e}^{ikx} + b {\rm e}^{i(k-1)x}\)
+|a|^2 b {\rm e}^{i(k-1)x} + a|b|^2
{\rm e}^{ikx}\\
&\quad +a^2\overline b {\rm e}^{i(k+1)x} +\overline a b^2 {\rm e}^{i(k-2)x}. 
\end{align*}
Leaving out the new harmonics (the last two exponentials) generated
both by band interaction and nonlinear effects, and identifying the
coefficients of 
${\rm e}^{ikx}$ and ${\rm e}^{i(k-1)x}$, we come up with:
\begin{equation*}
  \left\{
    \begin{aligned}
      i\d_t a& = \frac{1}{2}\(k-\alpha t\)^2a + \frac{v}{2}b +
      \epsilon \(|a|^2+2|b|^2\)a,\\
i\d_t b& = \frac{1}{2}\(k-1-\alpha t\)^2 b + \frac{v}{2}a +
      \epsilon \(2|a|^2+|b|^2\)b.
    \end{aligned}
\right.
\end{equation*}
We notice the identity $\d_t\(|a|^2+|b|^2\)=0$, so we can write $
|a|^2+|b|^2 =m_0^2>0$ (in \cite{BiQi00}, $m_0=1$).  
Now recalling the numerical value $k=1/2$ and expanding the
squares, we have
\begin{equation*}
  \left\{
    \begin{aligned}
      i\d_t a& = \frac{1}{8}a -\frac{\alpha t}{2} a +\frac{(\alpha
        t)^2}{2}a+ \frac{v}{2}b +
      \epsilon \(m_0^2+|b|^2\)a,\\
i\d_t b& = \frac{1}{8}b +\frac{\alpha t}{2}b + \frac{(\alpha
        t)^2}{2}b + \frac{v}{2}a +
      \epsilon \(m_0^2+|a|^2\)b.
    \end{aligned}
\right.
\end{equation*}
The above system becomes
\begin{equation*}
  i\d_t 
  \begin{pmatrix}
    a\\
b
  \end{pmatrix}
 = \(E_0+\frac{(\alpha
        t)^2}{2}\)
 \begin{pmatrix}
    a\\
b
  \end{pmatrix}
+\frac{1}{2}
\begin{pmatrix}
-\alpha t & v \\ v & \alpha t
\end{pmatrix}
\begin{pmatrix}
    a\\
b
  \end{pmatrix}
+\frac{\epsilon}{2}\, {\rm diag}\( |b|^2,|a|^2\) 
\begin{pmatrix}
    a\\
b
  \end{pmatrix},
\end{equation*}
with
\begin{equation*}
  E_0 = \frac{1}{8} +m_0^2 \epsilon . 
\end{equation*}
Using finally the gauge transform
\begin{equation*}
  \begin{pmatrix}
    u_1\\
u_2
  \end{pmatrix}
=
{\rm e}^{-iE_0t -i\alpha^2 t^3/6}\begin{pmatrix}
    a\\
b
  \end{pmatrix},
\end{equation*}
we end up with
\begin{equation*}
  i\d_t 
  \begin{pmatrix}
    u_1\\
u_2
  \end{pmatrix}
 = \frac{1}{2}
 \begin{pmatrix}
-\alpha t & v \\ v & \alpha t
\end{pmatrix}
\begin{pmatrix}
    u_1\\
u_2
  \end{pmatrix}
+\frac{\epsilon}{2}\,{\rm diag}( |u_2|^2,|u_1|^2) 
\begin{pmatrix}
    u_1\\
u_2
  \end{pmatrix},
\end{equation*}
which is of the form
\eqref{eq:modelphys1}--\eqref{eq:modelphys2} or~\eqref{eq:NLLZ2} via the change of variable $s=t\sqrt{\alpha\over 2} $; we then have $G=(2\alpha)^{-1/2}v$ and $F(u)=\eps(2\alpha)^{-1/2}{\rm diag}(|u_2|^2,|u_1|^2)$. Note that the only
approximation that we have made in this computation consists in
neglecting the other harmonics than ${\rm e}^{\pm ix/2}$, and no
linearization was performed, in contrast to the computations of
\cite{BiQi00}. 
\smallbreak
Theorem~\ref{theo:scat} gives  asymptotics for the profiles $a$ and
$b$ of the ansatz~\eqref{ansatz1} as $t$ goes to
$\pm\infty$. Theorem~\ref{theo:dev} gives an information on the
profiles $(a_+,b_+)$ for  $t\sim +\infty$ in terms of $(a_-,b_-)$,
those for $t\sim -\infty$. For example if for $t\sim-\infty$, we have
$(a_-,b_-)=(a,0)$, then the profiles for $t\sim +\infty$ are related
via the Landau-Zener transition coefficient~\eqref{LZcoeft} for
$z=(2\alpha)^{-1/2}v$: at leading order in $\delta$, they satisfy 
$$|a_+|^2={\rm e}^{-\pi v^2/(2\alpha)}|a|^2 \;\;{\rm and}\;\;
|b_+|^2=\left(1-{\rm e}^{-\pi v^2/( 2 \alpha)}\right)|a|^2.$$

\subsection{Condensate in a double-well potential}

As suggested in \cite{BiQi00}, and further developed in \cite{Kh10},
nonlinear Landau--Zener tunneling may be realized in a double-well
potential. We present a derivation which is
different from the one presented in the above references, even on a
former level.  Rigorous details are presented in the appendix. Consider
\begin{equation}
  \label{eq:nlswell}
  i \frac{\d \psi}{\d t} +\frac{1}{2}\frac{\d^2 \psi}{\d x^2}  =
  W(t,x) \psi +\epsilon |\psi|^2\psi.
\end{equation}
The potential $W$ is of the form
\begin{equation}\label{eq:double-well-V}
  W(t,x) = V_s(x)+\kappa t V_a(x),
\end{equation}
where $V_s$ is a symmetric double-well potential and
$V_a$ is antisymmetric.  
The main point to be aware
of is that the lowest two 
eigenvalues $\lambda_+<\lambda_-$ of the Hamiltonian
$-\frac{1}{2} \d_x^2 + V_s$ are non-degenerate,  with associated
eigenfunctions $\varphi_\pm$ (see Appendix~\ref{sec:justif} for details). The two exponential functions ${\rm e}^{\pm i
  x/2}$ of the above model are then replaced by the so-called
single-well states
\begin{equation}\label{eq:double-well-single}
  \varphi_L = \frac{1}{\sqrt 2}\(\varphi_+-\varphi_-\),\quad
  \varphi_R = \frac{1}{\sqrt 2}\(\varphi_++\varphi_-\). 
\end{equation}
We note that this approach can be generalized to a multidimensional
framework, in the spirit of \cite{Sa05} (see
Appendix~\ref{sec:justif}). We sketch the 
computation in the simplest~1D case though, to emphasize the
differences with the optical lattice case. 
We shall use mostly two properties of $\varphi_L$ and $\varphi_R$,
described more precisely in Appendix~\ref{sec:justif}:
\begin{itemize}
\item $\varphi_R(-x) = \varphi_L(x)$.
\item The product $\varphi_L\varphi_R$ is negligible, because $\varphi_L$ and
  $\varphi_R$ are localized at the 
two distinct minima of $V_s$. 
\end{itemize}
Seek $\psi$ of the form
\begin{equation*}
  \psi(t,x) = a_L(t)\varphi_L(x) + a_R(t)\varphi_R(x).
\end{equation*}
Denoting
\begin{equation*}
  \Omega = \frac{\lambda_-+\lambda_+}{2},\quad \omega =
  \frac{\lambda_--\lambda_+}{2}, 
\end{equation*}
we compute:
\begin{align*}
  i\d_t \psi& = i\dot a_L \varphi_L + i\dot a_R \varphi_R,\\
-\frac{1}{2}\d_x^2 \psi+V_s \psi& =  a_L\(\Omega
\varphi_L-\omega\varphi_R\) + 
a_R\(\Omega \varphi_R- \omega \varphi_L\)\\ 
&  = (\Omega a_L-\omega a_R)\varphi_L+(\Omega_R-\omega a_L)\varphi_R,\\ 
V_a \psi & = a_L V_a \varphi_L + a_R V_a \varphi_R,\\
|\psi|^2\psi & = \(|a_L|^2 |\varphi_L|^2 + |a_R|^2
|\varphi_R|^2+2\RE\(\overline a_L a_R\overline
\varphi_L\varphi_R\)\)\(a_L\varphi_L+a_R\varphi_R\). 
\end{align*}
By
integrating in space and neglecting the product $\varphi_L\varphi_R$, we get:
\begin{equation*}
\left\{
  \begin{aligned}
    i\d_t a_L& = \Omega a_L-\omega a_R +\kappa t \gamma_L a_L +
    \epsilon_L|a_L|^2
    a_L,\\
i\d_t a_R& = \Omega a_R- \omega a_L +\kappa t \gamma_R a_R +
    \epsilon_R|a_R|^2
    a_R,
  \end{aligned}
\right.
\end{equation*}
with
\begin{equation*}
  \gamma_L = \int_\R V_a \varphi_L^2,\quad \epsilon_L=\epsilon\int_\R
  \varphi_L^4,
 \quad\gamma_R = \int_\R V_a \varphi_R^2,\quad \epsilon_R=\epsilon\int_\R
  \varphi_R^4.
\end{equation*}
By symmetry, $\gamma_L = -\gamma_R$, $\epsilon_L=\epsilon_R=:\delta$, so if we
set $\alpha = \kappa \gamma_L$, we come up with:
\begin{equation*}
\left\{
  \begin{aligned}
    i\d_t a_L& =  \Omega a_L-\omega a_R +\alpha t  a_L +
    \delta |a_L|^2
    a_L,\\
i\d_t a_R& =  \Omega a_R- \omega a_L-\alpha t  a_R +
    \delta|a_R|^2
    a_R.
  \end{aligned}
\right.
\end{equation*}
Using  the gauge transform
\begin{equation*}
  \begin{pmatrix}
    u_1\\
u_2
  \end{pmatrix}
=
\begin{pmatrix}
    a_L {\rm e}^{i\Omega t}\\
a_R {\rm e}^{i\Omega t}
  \end{pmatrix},
\end{equation*}
we find:
\begin{equation*}
 i\d_t 
  \begin{pmatrix}
    u_1\\
u_2
  \end{pmatrix}
 = 
 \begin{pmatrix}
   \alpha t & -\omega \\
-\omega &-\alpha t
 \end{pmatrix}
\begin{pmatrix}
    u_1\\
u_2
  \end{pmatrix}
+\delta
\begin{pmatrix}
 |u_1|^2 & 0 \\
0 &|u_2|^2 
\end{pmatrix}
\begin{pmatrix}
    u_1\\
u_2
  \end{pmatrix}  ,
\end{equation*}
which is  of the form \eqref{eq:NLLZ2} via the change of variables
$s=\sqrt\alpha t$ with $F_j(u)=\alpha^{-1/2}|u_j|^2$ and
$G=-\omega\alpha^{-1/2}$. Therefore, Theorems~\ref{theo:scat}
and~\ref{theo:dev} yield similar results as in the preceding
subsection with a Landau-Zener coefficient ${\rm
  e}^{-\pi \omega^2/ \alpha}$.

\section{Preparation of the proof of Theorem~\ref{theo:scat}}
\label{sec:prelim}
  
\subsection{A useful reduction}
\label{sec:reduc}

For $u\in {\mathcal H}^2$, set
\begin{equation}\label{def:Mm}
  M(u) = {\delta\over 2} \left(F_1(u)+F_2(u)\right)\quad ;\quad 
m(u) = {\delta\over 2} \left(F_1(u)-F_2(u)\right), 
\end{equation}
where $M$ and $m$ are real-valued, and rewrite \eqref{eq:NLLZ2} as
\begin{equation}\label{syst:v}
  -i\d_s u = V(s)u +M(u)u +{\rm diag}\(m(u),-m(u)\)u = V\(s+m(u)\)u
  +M(u)u. 
\end{equation}
Since $M$ is a scalar, the last term can be absorbed by a gauge
transform:
\begin{equation}\label{eq:v}
v(s):= u(s){\rm e}^{-i\int_0^s M(u(\tau))d\tau}
\end{equation}
satisfies $\|v_j(s)\|_{\mathcal H}=\|u_j(s)\|_{\mathcal H}$ for $j\in\{1,2\}$ and
\begin{equation*}
 -i\d_s v = V\(s+m(u(s))\)v= \widetilde
 V\(s\)v, 
\end{equation*}
where we have introduced the notation $\widetilde V(s) =
V\(s+m(u(s))\)$.

\begin{remark}\label{rem:scatF1=F2}
Note  that if $F_1=F_2$, we have~$M=F_1$ and $m=0$ in~(\ref{def:Mm}); therefore, we are reduced to the standard linear equation 
\begin{equation*}
 -i\d_s v = V\(s\)v,
\end{equation*}
for which the scattering result is known, describing the large time
behavior of the solution $v$ at leading order (see  Lemma~7 in \cite{FG02}
or equivalently Lemma~11 in \cite{FG03}). In Proposition~5.5
of~\cite{FL08}, the remainder term is proved to be of order $\O(R^2s^{-1})$. 
 \end{remark} 
 
  In view of the physical models presented in Section~\ref{sec:deriv},
  we are mainly concerned with the situation where $F_1\not= F_2$, which
  is more involved mathematically.  
The first step of the large time study consists in proving the
following proposition.

\begin{proposition} \label{prop:u1} Let $R>0$. For any data $u_0$ such
  that  $\Theta_R u_0=u_0$, there exist two vectors of $(\R^+)^2 $,
  $\Omega=(\Omega_1,\Omega_2)$ and $A=(A_1,A_2)$ such that  and for  
$j\in\{1,2\}$,
\begin{align*}
&\| u_j(s)\|^2_{\mathcal H} = \Omega_j+ \O_R\(\frac{1}{s}\) \text{ as
}s\to +\infty,\\
& \| u_j(s)\|^2_{\mathcal H} =
A_j+\O_R\(\frac{1}{|s|}\) \text{ as
}s\to -\infty.
 \end{align*}
Moreover, 
  \begin{align*}
    &\int_s^{+\infty}
 \left(\| u_j(\tau)\|^2_{\mathcal H}-\Omega_j\right)d\tau=\O_R\(\frac{1}{s}\) \text{
   as }s\to +\infty,\\
 &\int_{-\infty}^s 
 \left(\| u_j(\tau)\|^2_{\mathcal H} -A_j\right)d\tau=\O_R\(\frac{1}{|s|}\)
\text{
   as }s\to -\infty.
 \end{align*}
\end{proposition}	
 Therefore, any $\omega\in {\mathcal H}^2$ such that
 $\|\omega_j\|_{\mathcal H}^2= \Omega_j$, $j=1,2$, satisfies
$F_j(\omega) = f_j (\Omega_1,\Omega_2)$ ($f_1$ and $f_2$ are the
functions in~\eqref{def:f}). Similarly, any $\alpha\in {\mathcal H}^2$ such that
 $\|\alpha_j\|_{\mathcal H}^2= A_j$, $j=1,2$, satisfies 
$F_j(\alpha) = f_j (A_1,A_2)$. We
shall however use the notations $\omega,\alpha$ until these vectors
are fully determined, after Proposition~\ref{prop:scatv} below. 
\smallbreak

Note that in view of~(\ref{eq:v}), it is equivalent to state
Proposition~\ref{prop:u1} for $u$ or for $v$. It will
be established by studying $v$ in Section~\ref{sec:scat}. In
particular, this proposition implies that 
the complete description of the function $v$ for large $s$ will induce
direct results on $u$. 
The analysis of the asymptotics of $v(s)$, as $s\rightarrow +\infty$,
is stated in the following proposition, which in turn is established
in  Section~\ref{sec:scat}. 
\begin{proposition}\label{prop:scatv}
There exist $(\nu_1,\nu_2)\in{\mathcal H}^2$ with
$\|\nu_j\|_{\mathcal H}^2=\Omega_j$, $j=1,2$, such that as
$s\rightarrow +\infty$,   
\begin{equation*}
 v_1(s) = {\rm e}^{i \varphi (s+m(\omega),GG^*)} \nu_1+
\O_R\(\frac{1}{s}\);\quad
 v_2(s) = {\rm e}^{-i \varphi (s+m(\omega),G^* G)} \nu_2 +
\O_R\(\frac{1}{s} \), 
\end{equation*}
and $(\mu_1,\mu_2)\in{\mathcal H}^2$ with
$\|\mu_j\|_{\mathcal H}^2=A_j$, $j=1,2$, such that as
$s\rightarrow -\infty$,   
\begin{equation*}
 v_1(s) = {\rm e}^{i \varphi (s+m(\alpha),GG^*)} \mu_1+
\O_R\(\frac{1}{|s|}\);\quad
 v_2(s) = {\rm e}^{-i \varphi (s+m(\alpha),G^* G)} \mu_2 +
\O_R\(\frac{1}{|s|} \). 
\end{equation*}
\end{proposition}
A similar statement holds in $-\infty$. 
\begin{proof}[Proposition~\ref{prop:scatv} implies Theorem~\ref{theo:scat}]
 Using  (\ref{eq:v}) and the second part of Proposition~\ref{prop:u1},
 we obtain  
\begin{align*}
u_1(s) & ={\rm e}^{i\int_0^s M(u(\tau))d\tau +i\varphi(s+m(\omega),GG^*)} \nu_1 +\O_R\(\frac{1}{s}\)\\
&={\rm e}^{is M(\omega) +i\varphi(s+m(\omega),GG^*)}{\rm e}^{i\int_0^{+\infty} \(M(u(\tau))-M(\omega)\)d\tau} \nu_1 +\O_R\(\frac{1}{s}\),\\
u_2(s) & ={\rm e}^{i\int_0^s M(u(\tau))d\tau -i\varphi(s+m(\omega),G^*G)} \nu_2 +\O_R\(\frac{1}{s}\)\\
&={\rm e}^{is M(\omega) -i\varphi(s+m(\omega),G^*G)}{\rm e}^{i\int_0^{+\infty} \(M(u(\tau))-M(\omega)\)d\tau} \nu_2 +\O_R\(\frac{1}{s}\).
\end{align*}
 Observing the identities
\begin{align*}
 \varphi(s+ m(\omega),\lambda )+s M(\omega) & =
\varphi(s,\lambda) + s\delta  F_1(\omega) +\frac{m(\omega)^2}{4}+
\O_R\(\frac{1}{s}\),\\ 
- \varphi(s+ m(\omega),\lambda )+s M(\omega)& =
-\varphi(s,\lambda)  + s\delta 
F_2(\omega)-\frac{m(\omega)^2}{4}+\O_R\(\frac{1}{s}\),  
\end{align*}
 we obtain Theorem~\ref{theo:scat} for $s\gg 1$ with 
$$\omega_1={\rm e}^{i\int_0^{+\infty}\left(M(u(\tau))-M(\omega)\right)d\tau + im(\omega)^2/4} \nu_1 ;
\quad\omega_2={\rm
  e}^{i\int_0^{+\infty}\left(M(u(\tau))-M(\omega)\right)d\tau -
  im(\omega)^2/4} \nu_2.$$ 
The case $s\to -\infty$ is similar.
\end{proof}
\begin{remark}
 If one expects the asymptotic behavior of $u$ to be described as in
Theorem~\ref{theo:scat}, then the first step (Proposition~\ref{prop:u1}) consists in
 deriving the asymptotic phase, while the role of the final step (Proposition~\ref{prop:scatv}) is to
 describe the amplitude. This strategy is similar to the one employed
 in e.g. \cite{HN98} to study the long range nonlinear scattering for the
 one dimensional cubic Schr\"odinger equation and for the Hartree equation. \\ 
\end{remark}

\subsection{Technical results}
\label{sec:technic}

Let us now introduce notations and state technical results that can be skipped by the reader at the first reading. 
We set
  \begin{equation}\label{def:JK}
  J:=\begin{pmatrix} 1 & 0 \\ 0 & -1\end{pmatrix}\;\;{\rm and}\;\;
  K:=V(0)=\begin{pmatrix} 0 & G \\ G^* & 0 \end{pmatrix},
  \end{equation}
  thus, $V(s)$ writes 
  $$V(s)=sJ+K.$$
  We denote by $\sigma(s)$ the function
  $$\sigma(s)=s+m(u(s))$$
 and we have  $$\widetilde V(s)=V(\sigma(s))\;\;{\rm and}\;\;\widetilde V(s)^2=\Lambda(s)^2,$$
  where $\Lambda(s)$ is the diagonal operator  
  $$  \Lambda(s)  = {\rm diag} \left( \sqrt{\sigma(s)^2+GG^*},
   \sqrt{\sigma(s)^2+G^*G} \right).$$
  For this reason, $\Lambda(s)$ appears like a diagonalisation of $V(\sigma(s))$, all the more that if we set 
  $$  \Pi^\pm(s)  = \frac{1}{2} \left({\rm Id} \pm
    \Lambda(s)^{-1}V(\sigma(s))\right),$$ 
  then we have the following properties.
  \begin{enumerate}
  \item $ \Pi^\pm (s)\widetilde V(s)=\widetilde V(s) \Pi^\pm(s) = \pm\Lambda(s) \Pi^\pm(s)=\pm\Pi^\pm(s)\Lambda(s)$.
\item $\Pi^+(s)+\Pi^-(s)={\rm Id}$.
\item $(\Pi^\pm(s))^*=\Pi^\pm(s)$.
\item  $\Pi^\pm(s)\Pi^\mp(s)=0$ and $(\Pi^\pm(s))^2=\Pi^\pm(s)$.
  \end{enumerate}
  The properties (2)--(4) show that $\Pi^\pm(s)$ are orthogonal
  projectors, and the property~(1) will play the role of a
  diagonalisation of the operator  $\widetilde V(s)$. The fact that $\Pi^\pm(s)$
  and $\widetilde V(s)$ commute with $\Lambda(s)$ is more general. In fact,
  $ \widetilde V(s)$ and $\Pi^\pm(s)$ are in the subset ${\mathcal A}$ of ${\mathcal L}({\mathcal H}^2)$ defined by:  $A\in{\mathcal A}$ if and only if there exist smooth functions $a$, $b$, $c$ and $d$ such that $A=A(a,b,c,d)$ with
  \begin{equation}\label{eq:form}
    A(a,b,c,d)=\begin{pmatrix} a(GG^*) & b(GG^*) G \\
  c(G^*G)G^* & d(G^*G)
\end{pmatrix}.
  \end{equation}
 A simple calculation
  shows that, because of the commutation property~\eqref{commutation},
  operators of~${\mathcal A}$ commutes with $\Lambda(s)$:
 $$\forall A\in{\mathcal A},\;\;\forall s\in\R,\;\;A\Lambda(s)=\Lambda(s)A .$$
 Besides ${\mathcal A}$ is an algebra, as shown by the following
 lemma which stems from straightforward computations 
 \begin{lemma}\label{lem:form} Let $a,b,c,d,a',b',c',d'\in{\mathcal
     C}^\infty(\R)$. We have 
   \begin{align*}
     A(a,b,c,d)^*&=A(\overline a, \overline c, \overline b , \overline
     d),\\
A(a,b,c,d)A(a',b',c',d')&=A(a'',b'',c'',d''),
   \end{align*}
 with 
 $$a''(\lambda)=a(\lambda)a'(\lambda)+\lambda b(\lambda ) c(\lambda),\;\; b''(\lambda)=a(\lambda)b'(\lambda)+b(\lambda)d'(\lambda),$$
 $$d''(\lambda)=d(\lambda) d'(\lambda) + \lambda c(\lambda)b(\lambda),\;\;c''(\lambda)=a'(\lambda)c(\lambda)+d(\lambda)c'(\lambda).$$
 \end{lemma}
Operators of ${\mathcal A}$ will be
 called {\it diagonal} if  
 $$A=\Pi^+(s)A\Pi^+(s)+\Pi^-(s)A\Pi^-(s),$$
 and {\it antidiagonal} if 
 $$A=\Pi^+(s)A\Pi^-(s)+\Pi^-(s)A\Pi^+(s).$$
 In particular,  $\widetilde V(s)$, $\Pi^+(s)$ and $\Pi^-(s)$ are diagonal  operators of ${\mathcal A}$; 
 on the other hand, the operators
 $\partial_s \Pi^+(s)$ and $\d_s
 \Pi^-(s)$ are  
 antidiagonal elements of ${\mathcal A}$.  Indeed, the relation
 $$\partial_s\Pi^+(s)=\partial_s((\Pi^+(s))^2)=\Pi^+(s)\partial_s\Pi^+(s)+\partial_s\Pi^+(s)\Pi^+(s)$$
 implies that $\Pi^\pm(s)\partial_s\Pi^+(s)\Pi^\pm(s)=0$ (and similarly for $\Pi^-(s)$ since $\partial_s\Pi^-(s)=-\partial_s\Pi^+(s)$).
 \smallbreak
 Antidiagonal operators have interesting properties   that we shall use
 later. Typically, they can be written as commutators with $\widetilde
 V(s)$: if $C(s)=\Pi^\mp(s) C(s)\Pi^\pm(s)$, then  we have
    $$C(s)=\pm \left[ B(s), \widetilde V(s)\right],$$  
with 
  \begin{equation}\label{eq:B}
  B(s)={1\over 2}\Lambda(s)^{-1} C(s),
   \end{equation}
   which also belongs to~${\mathcal A}$.
 Because of this property, we have the following lemma.
  \begin{lemma} \label{lem:offdiag} 
Let $v$ be a solution of~\eqref{syst:v}, $C(s)\in{\mathcal A}$ with
$C(s)=\Pi^+(s) C(s)\Pi^-(s)$, and  let $B(s)$ be associated  with $C(s)$ as
in~\eqref{eq:B}, then  
  \begin{align} \label{eq:offdiag}
 \< C(s)v(s)\;,\;v(s)\>_{{\mathcal H}^2}
  & ={1\over i}\frac{d}{ds} \<  B(s)
   v(s)\;,\;v(s)\>_{ {\mathcal H}^2} 
  +i
  \<\partial_sB(s)
    v(s)\;,\;v(s)\>_{{\mathcal H}^2}
    . 
  \end{align}
\end{lemma}

\begin{proof}
We write
\begin{align*}
\<C(s) v(s)\;,\; v(s)\>_{{\mathcal H}^2}& =
\<\left[B(s) \;,\; \widetilde V(s)
\right]v(s)\;,\;v(s)\>_{{\mathcal H}^2}\\ 
& =  \<B(s) \frac{1}{i}\partial_s v(s)\;,\; v(s)\>_{{\mathcal H}^2} - 
\< B(s)  v(s)\;,\; \frac{1}{i}\partial_s
  v(s)\>\\ 
& =  \frac{1}{i}\frac{d}{ds} \< B(s)
  v(s)\;,\;v(s)\>_{ {\mathcal H}^2} 
 -\frac{1}{i}
  \<\partial_sB(s)
    v(s)\;,\;v(s)\>_{{\mathcal H}^2}, 
  \end{align*}
and the lemma follows.
  \end{proof}
  
  Before closing this section, we gather in a lemma some estimates on the function $\sigma$  that we will use in the following.  We set 
  $$m(v(s))=\widetilde m\left(\| v_1(s)\|^2_{\mathcal H},\| v_2(s)\|_{\mathcal H}^2\right).$$
  
  \begin{lemma}\label{lem:sigma} Let $R>0$, for all initial data $u_0$ such that 
  $\Theta_R u_0=u_0$, we have 
  \begin{align}
  \label{dsigma1}
\quad  \dot \sigma(s) & =  1-2 \underline m(s)  \IM\<v_1(s),Gv_2(s)\>_{\mathcal H}=\O(R),\\
 \label{dsigma2}
 \quad \ddot\sigma(s) & =  4s\, \underline m(s) \RE\<v_1(s),Gv_2(s)\>_{\mathcal H}+\O(R^2),
  \end{align}
where $\dot \si(s)$ stands for $\d_s \si(s)$ and where $\underline m$  is the bounded function:
\begin{equation}\label{def:K}
\underline m(s)=\partial_1\widetilde m\left(\| v_1(s)\|^2_{\mathcal H},\| v_2(s)\|_{\mathcal H}^2\right)-\partial_2 \widetilde m \left(\| v_1(s)\|^2_{\mathcal H},\| v_2(s)\|_{\mathcal H}^2\right).
\end{equation}
  \end{lemma}
  
  \begin{proof}
Note first   that the functions $f_1$ and $f_2$ are  bounded on the unit ball, thus $m(u(s))$ and, for any multi-indices $\alpha\in\N^2$, $(\partial_1^{\alpha_1} \partial_2^{\alpha_2}m)(v(s))$, are   bounded, uniformly in $s$.
   Besides, we have the relations
   \begin{eqnarray}\label{eq:mborne}
\partial_s (\| v_1\|^2) & = & -\partial_s(\|v_2\|^2)  =  -2\IM\<v_1,Gv_2\>=\O(R),\\
\nonumber
\partial_s^2 (\|v_1\|^2) & = & 4s\RE\<v_1,Gv_2\>+2\(\|Gv_2\|^2-\|G^*v_1\|^2\).
\end{eqnarray}
  \end{proof}
 
\section{Existence of scattering states} \label{sec:scat}
 
 The proof of the existence of scattering states  consists of three steps.
 \begin{enumerate}
 \item We first prove that $\| \Pi^\pm v(s)\|_{{\mathcal H}^2} $ have a
   limit  as $s$ goes to $\pm\infty$. 
       \item We then deduce Proposition~\ref{prop:u1}.
       \item We finally prove Proposition~\ref{prop:scatv}, that is,
         the existence of scattering states for the function $v$. 
 \end{enumerate}
 Recall that Proposition~\ref{prop:scatv} implies Theorem~\ref{theo:scat} as explained in Section~\ref{sec:reduc}.

\subsection{Large time convergence of the norm}
\label{sec:largeCVnorm}

In this section, we prove an auxiliary result leading to Proposition~\ref{prop:u1}.

\begin{proposition}\label{prop:u+} Let $R>0$. Then for all initial data $u_0$ such that $u_0=\Theta_R u_0$, there exist $\Omega^\pm, A^\pm\ge 0$ such
  that, if we set $v^\pm= \Pi^\pm v$,
  \begin{align*}
    & \| v^\pm(s)\|^2_{{\mathcal H}^2}-\Omega^\pm=\O_R\left({1\over
        s^2}\right)\;\;{\rm 
      as}\;\;s\rightarrow +\infty,\\ 
 & \| v^\pm(s)\|^2_{{\mathcal H}^2}-A^\pm=\O_R\left({1\over
     s^2}\right)\;\;{\rm as}\;\;s\rightarrow  -\infty. 
  \end{align*}
\end{proposition}

\begin{proof} We consider the limit $s\rightarrow +\infty$ for the
  $+$~mode. The limit  $s\rightarrow -\infty$  and the case of the
  $-$~mode can be treated similarly.   
 We are going to prove that  $\| v^+(s)\|^2_{{\mathcal 
     H}\times{\mathcal H}}$ is a Cauchy sequence as $s\to +\infty$. 
     Note first that $v^\pm$ satisfies
     $$-i\partial_s v^\pm =\pm\Lambda(s) v^\pm + \partial_s \Pi^\pm(s) v.$$ 
Therefore, for $0<t<s$, 
$$
\|v^+(s)\|^2_{{\mathcal H}^2}-\|v^+(t)\|^2_{{\mathcal
    H}^2}= 2\RE \( \int_t^s  \<\Pi^+(\tau)\partial_s
\Pi^+(\tau) v(\tau),
    v(\tau)\>_{{\mathcal 
      H}^2} d\tau\right). 
$$
We are going to use  properties of the operator $
\Pi^+(s)\partial_s \Pi^+(s)$ that we gather in the next lemma where we
denote by $\IM C$ the skew adjoint part of the operator $C$: $\IM
C=(C-C^*)/2$.

\begin{lemma}\label{lem:dPi}
Let $C(s)= \Pi^+(s)\partial_s  \Pi^+(s)$ and
$B(s)={1\over 2}  \Lambda(s)^{-1} C(s)$. Then $C(s)$ and~$B(s)$ are
antidiagonal operators of~${\mathcal A}$ with
$C(s)=\O_R\left(1/s\right)$ and $B(s)=\O_R\left(1/ s^2\right)$   
 in ${\mathcal L}({\mathcal H}^2)$.
 Moreover, 
$$
\<\partial_s (\IM B(s))v(s),v(s)\> ={i\over 8s^3}
\underline m(s)
\partial_s \left(\(\RE\<v_1,Gv_2\>_{\mathcal
    H}\)^2\right)+\O_R(s^{-3}),
$$
where the bounded function $\underline m(s)$ is defined in~(\ref{def:K}).
 \end{lemma}

The proof of this lemma is postponed to the end of the section. 
Using Lemma~\ref{lem:offdiag}, we obtain
\begin{align*}
\RE  \int_t^s &\<C(\tau)
  v(\tau) \,,\; v(\tau)\>_{{\mathcal H}^2}
  d\sigma \\
&=\RE \left[{ i} \<B(\tau) v(\tau)\;,\;
  v(\tau)\>_{{\mathcal H}^2}\Big|_t^s-i\int_t^s \<\partial_sB(\tau)v(\tau)\;,\;
v(\tau)\>_{{\mathcal H}^2} d\tau\right]\\
& ={ i} \<\IM B(\tau)v(\tau)\;,\;
 v(\tau)\>_{{\mathcal H}^2}\Big|_t^s-i\int_t^s \<\partial_s\IM B(\tau)v(\tau)\;,\;
v(\tau)\>_{{\mathcal H}^2} d\tau\\
& = \O\({1\over t^2}\)-i\int_t^s \<\partial_s\IM B(\tau)v(\tau)\;,\;
v(\tau)\>_{{\mathcal H}^2} d\tau.
\end{align*}
and an integration by parts allows us to conclude that
$${\rm Re}  \int_t^s \<C(\tau)
v(\tau) \,,\; v(\tau)\>_{{\mathcal H}^2}
  d\tau=\O_R(s^{-2})$$
 since $\underline m(s)$ is bounded.  \end{proof}

\begin{proof}[Proof of Lemma~\ref{lem:dPi}]
The proof relies on the computation of $C(s)$ by observing (with the
notations of Section~\ref{sec:prelim})
$$\partial_s  \Pi^+(s)={1\over 2}\left(\dot \sigma(s) \Lambda(s)^{-1} J -\partial_s \Lambda(s)^{-1}\widetilde  V(s)\right).$$
In view of  
$$\partial_s  \Lambda(s)=\sigma(s)\dot\sigma(s)
\Lambda(s)^{-1},$$ 
and since 
$$ \Pi^+(s)\partial_s\Pi^+(s)= \Pi^+(s)\partial_s
\Pi^+(s)\Pi^-(s),$$
the operators $C(s)$ and $B(s)$ are antidiagonal and the estimates
$C(s)=\O_R(1/s)$ and $B(s)=\O_R(1/s^2)$ come from
Lemma~\ref{lem:sigma} and $\Lambda(s)=\O_R(1/s)$.\\ 

Let us now calculate $\partial_s\IM B(s)$.  We have
\begin{align*}
\Pi^+(s) \partial_s \Pi^+(s)\Pi^-(s) & =  {1\over 2}\dot \sigma(s)\Lambda(s)^{-1 }\Pi^+(s) J \Pi^-(s)\\
& =  {1\over 8} \dot \sigma(s)\Lambda(s)^{-1}\(J+\Lambda(s)^{-1}[ \widetilde V(s), J] - \Lambda(s)^{-2} \widetilde V(s) J\widetilde V(s)\).
\end{align*}
In view of 
$$[\widetilde V(s),J]=[K,J]=2 
\begin{pmatrix} 0 & -G \\ G^* & 0 \end{pmatrix},\;\;  \widetilde V(s)J\widetilde V(s)=\sigma(s)^2J+2\sigma(s)K+ 
{\rm diag} (-GG^* ,G^*G),$$
we obtain
$$C(s)={1\over 4}\dot\sigma(s)\( \Lambda(s)^{-3}\left( {\rm
    diag}(GG^*,G^*G)-\sigma(s)K\right)+ \Lambda(s)^{-2} \begin{pmatrix} 0 &
  -G\\ G^* & 0 \end{pmatrix}\),$$ 
hence
$$\IM C(s)={1\over 4} \dot\sigma(s)\Lambda(s)^{-2} \begin{pmatrix} 0 & -G\\ G^* & 0 \end{pmatrix}.$$
By Lemma~\ref{lem:sigma}, we obtain 
\begin{eqnarray*}
\partial_s \IM B(s) &=& {1\over 8} \ddot\sigma(s)\Lambda(s)^{-3} \begin{pmatrix} 0 & -G\\ G^* & 0 \end{pmatrix}+\O_R\left(1/s^3\right)\\
&=& {1\over 2s^2}\underline m(s) \RE\<v_1(s),Gv_2(s)\>_{\mathcal H}  \begin{pmatrix} 0 & -G\\ G^* & 0 \end{pmatrix}+\O_R\left(1/s^3\right)\end{eqnarray*}
where we have used $\Lambda(s)^{-3}-s^{-3}=\O_R(s^{-4})$, which stems from 
$$\forall m,\lambda\in\R^+,\;\;
((m+s)^2+\lambda)^{-3/2} - s^{-3}=\O\({m+\lambda \over s^4}\).$$
Finally, we obtain
$$
\<\partial_s \IM B(s) v(s),v(s)\> _{{\mathcal H}^2}=  {i\over 2s^2}\underline m(s) \RE\<v_1(s),Gv_2(s)\>_{\mathcal H}
 \IM\<v_1(s),Gv_2(s)\>_{\mathcal H}$$
and we conclude by observing 
$$\partial_s \left(\RE\<v_1(s),Gv_2(s)\>_{\mathcal H}\right)=-2s\IM\<v_1(s),Gv_2(s)\>_{\mathcal H}.$$
\end{proof}

\subsection{Analysis of the nonlinear term}
\label{sec:step2}
We are now going to prove Proposition~\ref{prop:u1}, which allows us to 
describe $F(v(s))$ as $s$ goes to $\pm\infty$.
 Note that,  for $|s|\gg 1$, we have
$$\Lambda(s)^{-1}\Theta_R=\frac{1}{|\sigma(s)|}\left({\rm Id} +\O_R\left({1\over
      s^2}\right)\right).$$
      We deduce  asymptotics for the operators $\Pi^\pm(s)$: for
$s\gg 1$, in ${\mathcal L}({\mathcal H}^2)$,
\begin{align}\label{Pi+asymptotic}
\Pi^+(s)\Theta_R & = \Theta_R E_1 +\frac{1}{2} \Lambda(s)^{-1}
K\Theta_R+ \O_R\left(\frac{1}{s^2}\right),\\ 
\nonumber
\Pi^-(s)\Theta_R & =   \Theta_RE_2 -\frac{1}{2} \Lambda(s)^{-1}
K\Theta_R+ \O_R\left(\frac{1}{s^2}\right), 
\end{align}
and for $s\ll -1$,  in ${\mathcal L}({\mathcal H}^2)$,\begin{align*}
\Pi^+(s)\Theta_R & =  \Theta_R E_2+\frac{1}{2} \Lambda(s)^{-1}
K\Theta_R+ \O_R\left(\frac{1}{s^2}\right),\\
\Pi^-(s)\Theta_R & =   \Theta_RE_1-\frac{1}{2} \Lambda(s)^{-1}
K\Theta_R+\O_R\left(\frac{1}{s^2}\right) ,
\end{align*}
where $K=V(0)$ has been defined in~(\ref{def:JK}) and 
$$E_1={\rm diag}(1,0)\;\;{\rm and}\;\; E_2={\rm diag}(0,1).$$
Since   $u_0=\Theta_R u_0$ yields that
for all $s\in\R$, $u(s)=\Theta_Ru(s)$ and $v(s)=\Theta_Rv(s)$ (by
Lemma~\ref{lem:localisation}), we have   
for all $R>0$,  as $s\rightarrow +\infty$,
\begin{align*}
\| v_1(s)\|_{\mathcal H}^2 &=  \|v^+(s)\|^2_{{\mathcal H}^2}+ 
\RE \< v^+(s) \;,\; \Lambda(s)^{-1} K\Theta_R
  v(s)\>_{{\mathcal H}^2}+
\O_R\left(\frac{1}{s^2}\right)\\ 
& =   \|v^+(s)\|^2_{{\mathcal H}^2}+ 
\RE \< K\Lambda(s)^{-1} \Pi^+(s)\Theta_R v(s) \;,\;
  v(s)\>_{{\mathcal H}^2}+
\O_R\left(\frac{1}{s^2}\right). 
\end{align*}
 In view of~\eqref{Pi+asymptotic} and of  the relation
 $$\| \Lambda(s)^{-1}K\Theta_R\|_{{\mathcal L}({\mathcal
     H}^2)}=\O_R(s^{-1}),$$ we have 
 $$\Pi^+(s)\Theta_R=E_1\Theta_R+\O_R(s^{-1}).$$
 This implies that we can write
\begin{align*}
K\Lambda(s)^{-1} \Pi^+(s)\Theta_R &= K\Lambda(s)^{-1} E_1\Theta_R+
\O_R\left(\frac{1}{s^2}\right)\\
& =  E_2 K\Lambda(s)^{-1} E_1\Theta_R+\O_R\left(\frac{1}{s^2}\right)\\
& = \Lambda(s)^{-1}  \Pi^-(s)
K\Theta_R\Pi^+(s)+\O_R\left(\frac{1}{s^2}\right), 
\end{align*}
where we have used $E_2K=KE_1$, $\Lambda(s)^{-1}
E_1=E_1\Lambda(s)^{-1}$ and the commutation properties of $V(s)$ and
$\Pi^\pm(s)$ with $\Lambda(s)^{-1}$ and~$\Theta_R$. We set 
\begin{equation}
  \label{eq:defg1}
  g_1(s)=\RE\< \Lambda(s)^{-1}\Pi^-(s) K \Theta_R\Pi^+(s) v(s) \;,\;
  v(s)\>_{{\mathcal H}^2}.
\end{equation}
We obtain
\begin{equation*}\label{g1etu1}
 g_1(s)=\O_R(s^{-1})\;\;{\rm and} \;\;
\| v_1(s)\|_{\mathcal H}^2 =  \|v^+(s)\|^2_{{\mathcal H}^2} + g_1(s)+\O_R\left({1\over s^2}\right).
\end{equation*}
Similarly, we obtain 
\begin{align}\label{eq:u1sinfty}
  \|v_2(s)\|^2_{{\mathcal H}}-\Omega^- & = 
  g_2(s)+\O_R\left(\frac{1}{s^2}\right) ,
  \end{align}
with $ g_2(s)  =\O_R(s^{-1})$ and,
 as~$s\rightarrow -\infty$, 
\begin{align*}
 \| v_1(s)\|^2_{{\mathcal H}}-A^-  =  \widetilde g_1(s)+
\O_R\left(\frac{1}{s^2}\right),\quad
  \| v_2(s)\|^2_{{\mathcal H}}-A^+  =  \widetilde g_2(s)+
\O_R\left(\frac{1}{s^2}\right),
  \end{align*}
with $ \widetilde g_j(s) =\O_R(|s|^{-1})$, $j=1,2$. 
This yields the first part of Proposition~\ref{prop:u1}, with
\begin{equation*}
(  \Omega_1,\Omega_2)= (\Omega^+,\Omega^-)\text{ and }(A_1,A_2)=
(A^-,A^+). 
\end{equation*}
Let us now prove that for $j\in\{1,2\}$ and $s\gg1$,
\begin{equation}\label{eq:gint}
\int_s^{+\infty} g_j(\tau)d\tau=\O_R(s^{-1} )\;\;{\rm and}
\int_{-\infty}^{-s} \widetilde 
  g_j(\tau)d\tau=\O_R(|s|^{-1}).
\end{equation}
We focus on $g_1$; the other assertions can be proved similarly. 
In that purpose, we study the operator 
$$\widetilde C(s)=\Lambda(s)^{-1}  \Pi^-(s) K\Theta_R\Pi^+(s).$$

\begin{lemma}\label{lem:dPibis}
The operators $\widetilde C(s)=\Lambda(s)^{-1}  \Pi^-(s) K\Theta_R\Pi^+(s)$ and $\widetilde B(s)={1\over 2} \Lambda(s)^{-1} \widetilde C(s)$ are antidiagonal operators of~${\mathcal A}$ which satisfy
\begin{equation}\label{eq:C(s)bis}
\IM \widetilde C(s)={\sigma(s)\over 2} \Lambda(s)^{-2}\Theta_R \begin{pmatrix} 0 & -G\\ G^* & 0 \end{pmatrix},
\end{equation}
and  $\widetilde C(s)=\O_R\left(1/s\right),\;\;\widetilde
B(s)=\O_R\left(1/s^2\right),\;\; \partial_s \widetilde
B(s)=\O_R\left(1/ s^3\right)$ in ${\mathcal L}({\mathcal H}^2)$.  
\end{lemma}

Then the proof is straightforward: by
Lemma~\ref{lem:offdiag}, we write
$$
g_1(s)=\RE\left(
\frac{1}{i} \frac{d}{ds} \<  \widetilde B(s) 
v(s)\;,\;v(s)\>_{{\mathcal H}^2 }\right)+\RE\left(i \<\partial_s \widetilde B(s)  v(s)\;,\;v(s)\>_{{\mathcal H}^2}\right),
$$
and \eqref{eq:gint} follows.
It remains to prove Lemma~\ref{lem:dPibis}.

\begin{proof}[Proof of Lemma~\ref{lem:dPibis}]
 We write 
 $$\widetilde C(s)={1\over 4} \Lambda(s)^{-1} \(K+\Lambda(s)^{-1} [K,\widetilde V(s)]-\Lambda(s)^{-2} \widetilde V(s)K\widetilde V(s)\)\Theta_R.$$
 In view of 
 $$[K,\widetilde V(s)]=2\sigma(s) \begin{pmatrix} 0 & -G \\ G^* &
   0\end{pmatrix}, $$
 and 
 $$\widetilde V(s)K\widetilde V(s)=-\sigma(s)^2K+2\sigma(s)\,{\rm diag}(GG^*,-G^*G)+ \begin{pmatrix} 0 & GG^*G \\ G^*GG^* & 0\end{pmatrix},$$
 we obtain 
 $$\displaylines{
 \widetilde C(s)={1\over 4} \Lambda(s)^{-1} \Bigl( -2 \sigma(s)\Lambda(s)^{-2} {\rm diag}(GG^*,-G^*G) +(1+\sigma(s)^2\Lambda(s)^{-2})K \hfill\cr\hfill
 - \Lambda(s)^{-2} 
 \begin{pmatrix} 0 & GG^*G \\ G^*G G^* & 0 \end{pmatrix} +2\sigma(s) \Lambda(s)^{-1} 
 \begin{pmatrix} 0 & -G \\ G^* & 0 \end{pmatrix}\Bigr) \Theta_R\cr}$$
 hence~(\ref{eq:C(s)bis}) and the other properties of the lemma follow from the observation $\|Ku(s)\|_{{\mathcal H}^2}=\O(R)$ and from Lemma~\ref{lem:sigma}.
\end{proof}

\subsection{Proof of Proposition~\ref{prop:scatv}}
\label{sec:prscat}

\begin{proof}
In view of the preceding computations, we have
\begin{equation*}
  m(v(s))=m(\omega) +g(s) +\O_R\(s^{-2}\),
\end{equation*}
where $\omega\in \mathcal H^2$ is such that $F_j(\omega)
=f_j(\Omega^+,\Omega_-)$,  with
\begin{equation}\label{eq:asymg}
  g(s)=\O_R(s^{-1} ) \;\;{\rm and} \;\; \int_s^{+\infty}
    g(\tau)d\tau =\O_R(s^{-1}). 
\end{equation}
Set $t=s+m(\omega)$ and $\check v(t)= v(s)$: $\check v$
solves
\begin{equation}\label{eq:checku}
  -i\d_t \check v = V\( t +
  m\(v(t-m(\omega))\)-m(\omega)\)\check v= V\(t+\check
  g(t)\)\check v, 
\end{equation}
where $\check g$ depends on $v$ and satisfies the same asymptotic estimates
\eqref{eq:asymg} as $g$. 
Consider the modified asymptotic phase
\begin{equation*}
  \check \varphi (t,\lambda) = \frac{t^2}{2}+\frac{\lambda }{2}\ln|t|
  -\int_t^{+\infty} \check g(\tau)d\tau,
\end{equation*}
and set, for $n\in \N$, 
\begin{equation*}
  \check v_{\rm app}^n(t)= \( {\rm e}^{i \check\varphi (t,GG^*)} \nu_1^n
  -\frac{G}{2t} {\rm e}^{-i\check \varphi (t,G^*G)}\nu_2^n,{\rm
    e}^{-i\check \varphi (t,G^*G)}\nu_2^n + \frac{G^*}{2t} {\rm e}^{i
    \check\varphi (t,GG^*)} \nu_1^n\), 
\end{equation*}
where $\nu_1^n$ and $\nu_2^n$ are to be fixed later. We compute
\begin{align*}
  -i\d_t \check v_{1,\rm app}^n & = \(t+\check g(t)\)  \check v_{1,\rm
    app}^n + G  \check v_{2,\rm app}^n +\O_R\( t^{-2}\),\\
-i\d_t \check v_{2,\rm app}^n & = -\(t+\check g(t)\)  \check v_{2,\rm
    app}^n + G^*  \check v_{1,\rm app}^n +\O_R\( t^{-2}\),
\end{align*}
hence
\begin{equation}\label{eq:checkuapp}
  -i\d_t \check v_{\rm app}^n =  V\(t+\check
  g(t)\)\check v_{\rm app}^n+\O_R\( t^{-2}\). 
\end{equation}
Now we fix $\nu_j^n$ so that $\check v_{\rm app}^n$ is close to $\check
v$ for $t=n$:
\begin{equation*}
  \nu_1^n = {\rm e}^{-i\check \varphi (n,GG^*)} \check v_1(n)\quad ;\quad 
\nu_2^n = {\rm e}^{i\check \varphi (n,G^*G)} \check v_2(n),
\end{equation*}
whence $\| \check v(n) -  \check v_{\rm app}^n(n)\|_{\mathcal
  H^2}=\O_R(n^{-1})$. Subtracting \eqref{eq:checkuapp} from
\eqref{eq:checku}, the energy estimate yields 
\begin{equation*}
\frac{d}{dt}\|\check v(t) -  \check v_{\rm app}^n  (t)\|_{\mathcal
  H^2} = \O_R\( t^{-2}\).
\end{equation*}
We infer, for $t\ge n$,
\begin{equation*}
  \|\check v(t) -  \check v_{\rm app}^n  (t)\|_{\mathcal
  H^2}=\O_R\( n^{-1}\),
\end{equation*}
and
\begin{equation*}
\check v_1(t) = {\rm e}^{i \check\varphi (t,GG^*)} \nu_1^n +
\O_R\(\frac{1}{t} +\frac{1}{n}\);\quad
\check v_2(t) = {\rm e}^{-i \check\varphi (t,G^* G)} \nu_2^n +
\O_R\(\frac{1}{t} +\frac{1}{n}\). 
\end{equation*}
Back to $v$, this information yields, since $v(s) = \check
v(s+m(\omega))$,
\begin{align*}
 v_1(s) &= {\rm e}^{i \varphi (s+m(\omega),GG^*) }\nu_1^n +
\O_R\(\frac{1}{s} +\frac{1}{n}\),\\ 
 v_2(s) &= {\rm e}^{-i \varphi (s+m(\omega),G^* G)} \nu_2^n +
\O_R\(\frac{1}{s} +\frac{1}{n}\),
\end{align*}
where we have used $\check\varphi(s,\lambda)=\varphi(s,\lambda)+\O_\lambda(1/s)$ by Proposition~\ref{prop:u1}.
We infer in particular
\begin{align*}
  \left| \|v_1(s)\|_{\mathcal H}^2  - \|\nu_1^n\|_{\mathcal
      H}^2\right| +  \left| \|v_2(s)\|_{\mathcal H}^2  - \|\nu_2^n\|_{\mathcal
      H}^2\right| \le C_R\(\frac{1}{n}+\frac{1}{s}\).
\end{align*}
Taking the $\limsup$ as $s\to +\infty$ yields, in view of
Proposition~\ref{prop:u+} and \eqref{g1etu1},
\begin{equation*}
  \left|\Omega^+- \|\nu_1^n\|_{\mathcal H}^2\right| + 
\left|\Omega^-- \|\nu_2^n\|_{\mathcal H}^2\right|=\O_R\(n^{-1}\).
\end{equation*}
Since the unit ball of ${\mathcal H}$ is compact, one can
extract a converging subsequence of $\nu^n_j$, whose limit $\nu_j$
satisfies $\|\nu_1\|^2_{\mathcal H}=\Omega^+
$ and $\|\nu_2\|^2_{\mathcal H}=\Omega^-$, as in Proposition~\ref{prop:scatv}. 
Uniqueness is straightforward. A similar argument allows us to define the
scattering states as $s\rightarrow -\infty$. 
\end{proof}

\section{Existence of wave operators}
\label{sec:waveop}

Given an asymptotic state $\omega=(\omega_1,\omega_2)\in {\mathcal
  H}^2$ such that $\omega=\Theta_R \omega$, consider the expected
asymptotic solution $u_{\rm app}$ provided by Theorem~\ref{theo:scat},
appearing in the statement of Proposition~\ref{prop:waveop} (up to the
factor $\phi$):
\begin{equation*}
  u_{1,{\rm app}}(s) = {\rm e}^{is\delta
    F_1(\omega)+i\varphi(s,GG^*)}\omega_1;\quad
u_{2,{\rm app}}(s) = {\rm e}^{is\delta
    F_2(\omega)-i\varphi(s,G^*G)}\omega_2.
\end{equation*}
For $n\in \N$, let $u^n=\Theta_R u^n$ be the solution to
\eqref{eq:NLLZ2} such that 
\begin{equation}\label{eq:unCI}
  u^n_{\mid s=n} = u_{{\rm app}\mid s=n}. 
\end{equation}
In view of Lemma~\ref{lem:existence}, for all $n\in \N$, $u^n$ is
defined globally, and $u^n\in C(\R;{\mathcal H}^2)$. Its initial value
$u^n_{\mid s=0}$ is a sequence of the unit ball (in view of
\eqref{eq:conserv}) of ${\mathcal H}^2$. Since the dimension of
$\mathcal H$ is finite, up to extracting a subsequence, we may assume
that 
\begin{equation*}
  u^n_{\mid s=0}\Tend n \infty {\uu}_0 \quad \text{in }{\mathcal H}^2. 
\end{equation*}
We denote by ${\uu}$ its evolution under \eqref{eq:NLLZ2}. 
Theorem~\ref{theo:scat} provides a unique $\widetilde
\omega\in {\mathcal H}^2$ such that, as $s\to +\infty$,
\begin{equation}
  \label{eq:uuasym}
  \begin{aligned}
   \uu_1(s) &=  {\rm e}^{is\delta F_1(\widetilde\omega) +
  i\varphi(s,GG^*)}\widetilde\omega _1+\O_R\(\frac{1}{s}\),\\
 \uu_2(s)&={\rm e}^{is\delta F_2(\widetilde\omega)
   -i\varphi(s,G^*G)}\widetilde\omega_2+\O_R\(\frac{1}{s}\). 
  \end{aligned}
\end{equation}
In the same spirit as to prove the existence of scattering states, we
first study the norms. 
\begin{lemma}\label{lem:normwaveop}
   We have $\|\omega_j\|_{\mathcal H}= \|\widetilde
   \omega_j\|_{\mathcal H}$, $j=1,2$. 
 \end{lemma}
 \begin{proof}
   Suppose that the result were not true, say $\|\omega_1\|_{\mathcal
     H}\not = \|\widetilde   \omega_1\|_{\mathcal H}$. We have seen in
   Section~\ref{sec:scat} that there exists $C_R$ independent of $u$
   such that
   \begin{equation*}
     \left| \|\uu_1(s)\|_{\mathcal H} - \|\widetilde
       \omega_1\|_{\mathcal H}\right|\le \frac{C_R}{s},\quad
     \forall s\ge 1.
   \end{equation*}
By assumption, there exists $\eps>0$ and $s_0>11$ such that
\begin{equation*}
  \left| \|\uu_1(s)\|_{\mathcal H} - \|
       \omega_1\|_{\mathcal H}\right|\ge \eps,\quad
     \forall s\ge s_0.
\end{equation*}
Form Theorem~\ref{theo:scat}, for all $n\in \N$, there exists
$\omega^n\in {\mathcal H}^2$ such that
\begin{equation}
  \label{eq:unfort}
  \begin{aligned}
     u_1^n(s)-{\rm e}^{is\delta F_1(\omega^n) +
  i\varphi(s,GG^*)}\omega^n_1 &= \O_R\(\frac{1}{s}\), \\
u_2^n(s)-{\rm e}^{is\delta F_2(\omega^n) -
  i\varphi(s,GG^*)}\omega^n_2 &= \O_R\(\frac{1}{s}\), 
  \end{aligned}
\end{equation}
hence
\begin{equation*}
     \left| \|u_1^n(s)\|_{\mathcal H} - \|
       \omega_1^n\|_{\mathcal H}\right|\le \frac{C_R}{s},\quad
     \forall s\ge 1,
   \end{equation*}
where $C_R$ does not depend on $n$. This yields
\begin{equation*}
  \|\omega_1^n\|_{\mathcal H}= \|\omega_1\|_{\mathcal
    H}+\O_R\(\frac{1}{n}\),
\end{equation*}
and there exists $C_R$ such that
\begin{equation}\label{eq:ineqwaveop}
 \left| \|u_1^n(s)\|_{\mathcal H} - \|
       \omega_1\|_{\mathcal H}\right|\le
     C_R\(\frac{1}{s}+\frac{1}{n}\),  \quad     \forall s\ge 1.
\end{equation}
Fix $s\ge s_0$ such that $CR^2/s<\eps/2$. Notice that \eqref{eq:NLLZ2} is
locally well-posed, uniformly on all compact intervals. Indeed, if
$u_a$ and $u_b$ are two solutions of \eqref{eq:NLLZ2}, subtracting the
two equations, using \eqref{eq:conserv} and the energy estimate for
$u_a-u_b$, Gronwall lemma yields, for some universal constant $C$,
\begin{equation*}
  \|u_a(s)-u_b(s)\|_{{\mathcal H}^2}\le
  \|u_a(s_0)-u_b(s_0)\|_{{\mathcal H}^2}{\rm e}^{C|s-s_0|}. 
\end{equation*}
In particular, $u_1^n(s)\to \uu_1(s)$
as $n\to \infty$ (recall that $s$ is fixed).  Taking the $\limsup$ in
$n$ in \eqref{eq:ineqwaveop} yields, 
\begin{equation*}
  \left| \|\uu_1(s)\|_{\mathcal H} - \|
       \omega_1\|_{\mathcal H}\right|\le
     \frac{C_R}{s}\le \frac{\eps}{2},
\end{equation*}
hence a contradiction. 
 \end{proof}
To conclude, we go back to \eqref{eq:unfort} which yields, for $s=n$
and in view of \eqref{eq:unCI},
\begin{align*}
 & \left\| {\rm e}^{in\delta F_1(\omega) }\omega_1 - {\rm e}^{in\delta
     F_1(\omega^n)}\omega^n_1\right\|_{\mathcal H} =
\O_R\(\frac{1}{n}\),\\
& \left\| {\rm e}^{in\delta F_2(\omega) }\omega_2 - {\rm
    e}^{in\delta F_2(\omega^n)}\omega^n_2\right\|_{\mathcal H} =
\O_R\(\frac{1}{n}\).
\end{align*}
Up to extracting a subsequence again, we may assume that
\begin{equation*}
 {\rm e}^{in\delta \(F_1(\omega) -F_1(\omega^n)\)}\Tend n \infty  {\rm
   e}^{i\phi_1}\quad\text{and}\quad
{\rm e}^{in\delta \(F_2(\omega) -F_2(\omega^n)\)}\Tend n \infty  {\rm
   e}^{i\phi_2},
\end{equation*}
for some $\phi_1,\phi_2\in [0,2\pi)$. We infer 
\begin{equation}\label{eq:appshift}
  \widetilde \omega_1 = {\rm e}^{i\phi_1}\omega_1\quad\text{and}\quad
\widetilde \omega_2 = {\rm e}^{i\phi_2}\omega_2.
\end{equation}
Set
\begin{equation*}
  \phi =\frac{\phi_1-\phi_2}{2},\quad \Phi =
  \frac{\phi_1+\phi_2}{2}. 
\end{equation*}
We see that by construction, $\phi=0$ if $F_1=F_2$. Let
\begin{equation*}
  u(s) = {\rm e}^{-i \Phi} \uu(s). 
\end{equation*}
Since $F$ is gauge invariant, $u$ solves \eqref{eq:NLLZ2} with
$u_0={\rm e}^{-i \Phi} \uu_0$. In view of \eqref{eq:uuasym},
\eqref{eq:appshift} and the definition of $\Phi$, we also have
\begin{align*}
 u_1(s) &=  {\rm e}^{is\delta F_1(\widetilde\omega) +
  i\varphi(s,GG^*)-i\Phi}\widetilde\omega _1+\O_R\(\frac{1}{s}\)
= {\rm e}^{is\delta F_1(\omega) +
  i\varphi(s,GG^*)+i\phi}\omega _1+\O_R\(\frac{1}{s}\),\\
 u_2(s)&={\rm e}^{is\delta F_2(\widetilde\omega)
   -i\varphi(s,G^*G)-i\Phi}\widetilde\omega_2+\O_R\(\frac{1}{s}\)
={\rm e}^{is\delta F_2(\omega)
   -i\varphi(s,G^*G)-i\phi}\omega_2+\O_R\(\frac{1}{s}\),
\end{align*}
hence the result. Finally, uniqueness for such a solution $u$ for a
fixed $\phi$ stems from Theorem~\ref{theo:scat}.

%%%%%%%%%%%%%%%%%%%%%%%%%%%%%%%%%%%%%%%%%%%%%%%%%%%%%%%%

\section{Analysis of the Scattering Operator}
\label{sec:dev}

We now prove Theorem~\ref{theo:dev} and we take into account the
dependence of the nonlinear term with respect to the
parameter~$\delta$. We begin with the particular case $F_1=F_2$, which
turns out to be fairly easy, and then turn to the general case.

\subsection{The case $F_1=F_2=\underline F$}
\label{sec:scatscal}

Recall that we denote by $u^{\rm lin}$ the solution to 
\begin{equation*}
  \frac{1}{i}\d_s u^{\rm lin} = V(s)u^{\rm lin};\quad u^{\rm
    lin}_{\mid s=0}=u_{\mid s=0}=u_0. 
\end{equation*}
In view of the identity
\begin{align}
\label{eq:u/ulin}  
u(s)&=\exp\( i\delta \int_0^s \underline F(u(\tau))d\tau\) u^{\rm
    lin}\\
&= \exp\(
i\delta\int_0^s(\underline F(u(\tau))-\underline F(u(\omega)))d\tau
+is\delta \underline F(\omega)\)u^{\rm lin} ,\notag
\end{align}
Theorem~\ref{theo:scat} yields
$$\omega  =  \exp\(i\delta\int_0^{+\infty} \left(\underline
  F(u(\tau))-\underline F(\omega)\right)d\tau\) \omega^{\rm lin}.$$ 
Similarly,
$$\alpha  =   \, \exp\(i\delta \int_0^{-\infty} \left(\underline
  F(u(\tau))-\underline F(\alpha)\right)d\tau\)\alpha^{\rm lin}.$$ 
In particular, we have  $\underline F(\omega)=\underline F(
\omega^{\rm lin})$ and $\underline F(\alpha)=\underline F(\alpha^{\rm
  lin})$ since  
\begin{equation}\label{egalitedesmodules}
\|\omega_j\|_{\mathcal H}=\|\omega_j^{\rm lin}\|_{\mathcal H}\;\;{\rm and}\;\; \|\alpha_j\|_{\mathcal H}=\|\alpha_j^{\rm lin}\|_{\mathcal H}\;\;{\rm for}\;\; j\in\{1,2\}.
\end{equation}
Besides, in view of \eqref{eq:u/ulin}, we obtain
 $\underline F(u)=\underline F(u^{\rm lin})$. We deduce 
\begin{align*}
\omega & =   \exp\(i\delta\int_0^{+\infty}
\left(\underline F(u^{\rm lin}(\tau))-\underline F(\omega^{\rm lin})\right)d\tau\)
\omega^{\rm lin}=
{\rm e}^{i\delta \Lambda^+}\omega^{\rm lin} ,\\ 
\alpha & = \exp\(i\delta \int_0^{-\infty}
\left(\underline F(u^{\rm lin}(\tau))-\underline F(\alpha^{\rm
    lin})\right)d\tau\) \alpha^{\rm lin} = 
 {\rm e}^{-i\delta \Lambda^-}\alpha^{\rm lin},
\end{align*}
where $\Lambda^+ =\int_0^{+\infty}
\left(F(u^{\rm lin}(\tau))-F(\omega^{\rm lin})\right)d\tau$ and 
$\Lambda^-= \int_{-\infty}^0
\left(F(u^{\rm lin}(\tau))-F(\alpha^{\rm lin})\right)d\tau$.
This yields the exact formula \eqref{eq:scatexpl}.

We readily infer the following expansion
 $$S_\delta=S^{\rm lin}+i \delta\, \left(\Lambda^++ \Lambda^-\right) S^{\rm lin}
 +\O(\delta^2\|\Lambda^+\|^2) +\O(\delta^2\|\Lambda^-\|^2).$$  
It remains to obtain an upper bound for $\Lambda^\pm$. We write
$$\Lambda^+=\int_0^{s_0}
\left(\underline F(u(s))-\underline F(\omega)\right)ds+
\int_{s_0}^{+\infty}\left(\underline F(u(s))-\underline F(\omega)\right)ds.$$ 
By the continuity of the functions $f$ (defined in~(\ref{def:f})) and
because $\|u_1\|^2_{\mathcal 
  H}+\|u_2\|^2_{\mathcal H}=\|\omega_1\|^2_{\mathcal
  H}+\|\omega_2\|^2_{\mathcal H}=1$,  we have
$$\left|\int_0^{s_0} \left(\underline F(u(s))-\underline
    F(\omega)\right)ds\right|\le 
2\left(\sup_{{\mathbf S}^1} |f|\right)\, s_0.$$ 
 Proposition~\ref{prop:u1}  yields a constant $C_R$ such that, 
\begin{equation}\label{eq:k=1}
\left|\int_{s_0}^{+\infty}
  \left(\underline F(u(s,z))-\underline
    F(\omega(z))\right)ds\right|\le C_Rs_0^{-1} . 
\end{equation}
We infer $\Lambda^+ =\O_R(1)$, hence Theorem~\ref{theo:dev} in
the case $F_1=F_2$.

\subsection{The general case}
\label{sec:scatgen}

Following for instance \cite{PG96} (see also \cite{CaGa09} for some
generalizations), the most natural method to obtain an asymptotic
expansion for $S_\delta^\phi$ should consist in decomposing the
nonlinear solution $u$ as
\begin{equation*}
  u = u^{(0)} +\delta r^\delta,
\end{equation*}
where $u^{(0)}$ solves the linear equation \eqref{eq:LZ} and behaves
like $u$ as $s\to -\infty$ (up to $\O(\delta)$). The goal would then
be to prove that the remainder $r^\delta$ is uniformly bounded for
$\delta\in [-\delta_0,\delta_0]$ for some $\delta_0>0$ possibly
small. However, if suitable in the case of dispersive equations (as in
\cite{PG96,CaGa09}), this approach does not seem to be successful in
the present framework, because the equation satisfied by the remainder
$r^\delta$ involves terms whose integrability is rather
delicate. Therefore, the approach that we present follows a different
path, and highly relies on structural properties related to
\eqref{eq:NLLZ2}. 
\smallbreak

Let $\alpha\in {\mathcal H}^2$, and $\phi$ provided by
Proposition~\ref{prop:waveop}. Proposition~\ref{prop:waveop} yields a
global solution $u$ to \eqref{eq:NLLZ2}, and Theorem~\ref{theo:scat}
provides an asymptotic state $\omega$ for $s\to +\infty$. Instead of
considering directly $u$, 
we will work with the function $v$ defined in~\eqref{eq:v} and
satisfying~\eqref{syst:v}. 
In terms of $v$, Proposition~\ref{prop:waveop} reads
\begin{align*}
 & v_1(s) = \ee^{is\delta F_1(\alpha) +i\varphi(s,GG^*)
    -isM(\alpha)+i\int_{-\infty}^0 \(
    M(u(\tau))-M(\alpha)\)d\tau+i\phi}\alpha_1 +\O_R\(\frac{1}{|s|}\),\\
& v_2(s) = \ee^{is\delta F_2(\alpha) -i\varphi(s,G^*G)
    -isM(\alpha)+i\int_{-\infty}^0 \(
    M(u(\tau))-M(\alpha)\)d\tau-i\phi}\alpha_2 +\O_R\(\frac{1}{|s|}\).
\end{align*}
In view of the expression of $\varphi$ and of the definition of $m$
and $M$, we can also write
\begin{align*}
 & v_1(s) = \ee^{i\varphi(s+m(\alpha),GG^*)
    +i\int_{-\infty}^0 \(
    M(u(\tau))-M(\alpha)\)d\tau+i\phi-im(\alpha)^2/2}\alpha_1
  +\O_R\(\frac{1}{|s|}\),\\ 
& v_2(s) = \ee^{-i\varphi(s+m(\alpha),G^*G)
    +i\int_{-\infty}^0 \(
    M(u(\tau))-M(\alpha)\)d\tau-i\phi+im(\alpha)^2/2}\alpha_2
  +\O_R\(\frac{1}{|s|}\). 
\end{align*}
Proposition~\ref{prop:waveop} in the \emph{linear} case $F_1=F_2=0$ 
yields a unique solution $v^{\rm lin}$ to 
\begin{equation*}
  \frac{1}{i}\d_s v^{\rm lin} = V\(s\)v^{\rm lin},
\end{equation*}
and such that, as $s\to -\infty$,
\begin{equation*}
  v_1^{\rm lin}(s) = \ee^{i\varphi(s,GG^*)}\alpha_1
  +\O_R\(\frac{1}{|s|}\),\quad 
v_2^{\rm lin}(s) = \ee^{-i\varphi(s,G^*G)}\alpha_2
  +\O_R\(\frac{1}{|s|}\). 
\end{equation*}
Therefore, $v^{\rm lin}(0)$ is the image of $\alpha$ under the action
of the \emph{linear} wave operator, and as $s\to +\infty$, $v^{\rm
  lin}$ has an asymptotic state $\omega^{\rm lin}= S^{\rm
  lin}(\alpha)$. 
Setting $v^-(s)=v^{\rm lin}(s+m(\alpha))$, we see that
\begin{equation*}
  \frac{1}{i}\d_s v^- = V\(s+m(\alpha)\)v^-,
\end{equation*}
and 
\begin{align*}
 & v_1^-(s) = \ee^{i\varphi(s+m(\alpha),GG^*)}\alpha_1
  +\O_R\(\frac{1}{|s|}\),\\
& v_2^-(s) = \ee^{-i\varphi(s+m(\alpha),G^*G)}\alpha_2
  +\O_R\(\frac{1}{|s|}\). 
\end{align*}
We note the identity
\begin{equation*}
  \|v_1(s)-v_1^-(s)\|_{\mathcal H} = \left\lvert \ee^{i\int_{-\infty}^0 \(
    M(u(\tau))-M(\alpha)\)d\tau+i\phi-im(\alpha)^2/2}- 1\right\rvert
  \|\alpha_1\|_{\mathcal H}+ \O_R\(\frac{1}{|s|}\).
\end{equation*}
The proof of Proposition~\ref{prop:waveop} yields $\phi=\O_R(\delta)$,
so 
\begin{equation*}
  \|v_1(s)-v_1^-(s)\|_{\mathcal H} = \O_R(\delta)+ \O_R\(\frac{1}{|s|}\).
\end{equation*}
We infer
\begin{equation*}
  \left\|v_1\(\frac{-1}{\delta}\)-v_1^-\(\frac{-1}{\delta}\)\right\|_{\mathcal
    H} = \O_R(\delta). 
\end{equation*}

We now use the following lemma the proof of which is postponed to the
end of the section.  
\begin{lemma}\label{lem:corrector1}
There exists $\kappa^-(s)$ such that, as $s\to -\infty$, $\kappa^-(s)=
\O_R(\delta s^{-2})$ and 
$$
{1\over i} \partial_s \left(v-v^- +\kappa^-(s)\right)=V(s+m(\alpha))\left(v-v^-+\kappa^-(s) \right)+\O_R(\delta s^{-2}).$$
\end{lemma}
 By integration, we infer
$$\forall s\le - 1,\;\;v(s)=v^-(s)+\O_R(\delta).$$
\smallbreak

The Fundamental Theorem of Calculus yields the identity
\begin{equation*}
  v^{\rm lin}(s) - v^-(s) = \int_{s+m(\alpha)}^s\d_s v^{\rm
    lin}(\tau)d\tau, 
\end{equation*}
and since $\d_s v^{\rm lin}$ is locally bounded (from the equation),
we infer
\begin{equation*}
 v^{\rm lin}(-1) - v^-(-1)=\O_R\(m(\alpha)\)   =\O_R(\delta).
\end{equation*}
On the other hand, since 
\begin{equation*}
  \frac{1}{i}\d_s \(v-v^{\rm lin}\) = V(s)\(v-v^{\rm lin}\)
  +m\(u(s)\)J v,
\end{equation*}
and since $v$ is bounded, the energy identity yields
\begin{equation*}
  \frac{d}{ds}\|v-v^{\rm lin}\|_{\mathcal H^2}^2=2 \IM\< m\(u(s)\)J
  v,v-v^{\rm lin}\> =\O_R(\delta\|v-v^{\rm lin}\|_{\mathcal H^2}). 
\end{equation*}
We infer
\begin{equation*}
  v(s) = v^{\rm lin}(s)+\O_R(\delta),\quad \forall s\in [-1,1].
\end{equation*}
With $\omega$ the (nonlinear) asymptotic state associated to $u$, we define
similarly
\begin{equation*}
  v^+(s) = v^{\rm lin}\(s+m(\omega)\). 
\end{equation*}
For the same reason as above, $v(1) = v^+(1)+\O_R(\delta)$. Proceding
like for $s<0$, we can find $\kappa^+(s)\in \R$ so that
\begin{equation*}
  {1\over i} \partial_s \left(v-v^+ +\kappa^+(s)\right)=V(s+m(\omega))\left(v-v^++\kappa^+(s)\right) +
\O_R\left({\delta\over s^2}\right),
\end{equation*}
and we infer
\begin{equation}\label{eq:SlinS2}
  v(s)=v^+(s)+\O_R(\delta), \quad \forall s\ge 1.
\end{equation}
From the linear scattering theory, we have, as $s\to +\infty$:
\begin{align*}
 & v_1^+(s) = \ee^{i\varphi(s+m(\omega),GG^*)}\omega_1^{\rm lin}
  +\O_R\(\frac{1}{s}\),\\ 
& v_2^+(s) = \ee^{-i\varphi(s+m(\omega),G^*G)}\omega_2^{\rm lin}
  +\O_R\(\frac{1}{s}\), 
\end{align*}
where $\omega^{\rm lin} = S^{\rm lin}(\alpha)$, and from
Theorem~\ref{theo:scat} (see also Proposition~\ref{prop:scatv}), we also have (like in the case $s\to -\infty$):
\begin{align*}
 & v_1(s) = \ee^{i\varphi(s+m(\omega),GG^*)
    -i\int_0^{+\infty} \(
    M(u(\tau))-M(\omega)\)d\tau-im(\omega)^2/2}\omega_1
  +\O_R\(\frac{1}{s}\),\\ 
& v_2(s) = \ee^{-i\varphi(s+m(\omega),G^*G)
    -i\int_0^{+\infty} \(
    M(u(\tau))-M(\omega)\)d\tau+im(\omega)^2/2}\omega_2
  +\O_R\(\frac{1}{s}\). 
\end{align*}
In view of these asymptotic identities and of \eqref{eq:SlinS2}, we
infer
\begin{equation*}
  \omega = \omega^{\rm lin}+\O_R(\delta),
\end{equation*}
which is the identity $S=S^{\rm lin}+\O_R(\delta)$ of
Theorem~\ref{theo:dev}. 

\bigbreak

It remains to prove the intermediary result:

\begin{proof}[Proof of Lemma~\ref{lem:corrector1}]
In view of~(\ref{eq:mborne}) and~(\ref{eq:gint}), we have 
\begin{eqnarray*}
\| v_1(s)\|_{\mathcal H}^2-\| \alpha_1\|_{\mathcal H}^2=-\left(
  \|v_2(s)\|_{\mathcal H}^2-\|\alpha_2\|_{\mathcal H}^2\right) & = &
\int_s^{-\infty} 2\IM
\<v_1(\tau),Gv_2(\tau)\> d\tau.
\end{eqnarray*}
Since 
$$\partial_s\langle v_1(s),Gv_2(s)\rangle=2is\langle v_1(s),Gv_2(s)\rangle +i\left(\|Gv_2(s)\|^2_{\mathcal H}-\|Gv_1(s)\|^2_{\mathcal H}\right),$$
an integration by parts yields
$$\displaylines{\qquad
\int_s^{-\infty} \<v_1(\tau),Gv_2(\tau)\>  =  {i\over 2s} \<v_1(s),Gv_2(s)\> \hfill\cr\hfill
 +\int_s^{-\infty} \left({1\over 2i\tau^2}\<v_1(\tau),Gv_2(\tau) \> -{1\over 2\tau} \left(\| Gv_2(\tau)\|^2_{\mathcal H}-\| G^*v_1(\tau)\|^2_{\mathcal H}\right)\) d\tau,
\qquad\cr}$$
whence
$$\displaylines{\qquad
\IM \int_s^{-\infty} \<v_1(\tau),Gv_2(\tau)\>  =\IM\left(  {i\over 2s} \<v_1(s),Gv_2(s)\> \right)\hfill\cr\hfill
 +\IM \int_s^{-\infty} {1\over 2i\tau^2}\<v_1(\tau),Gv_2(\tau) \>d\tau+\O_R(s^{-2}).\qquad\cr}$$
With one more integration by parts, we obtain
\begin{align*}
2\IM \int_s^{-\infty} \<v_1(\tau),Gv_2(\tau)\> &= {1\over s} \RE \<v_1(s),Gv_2(s)\> +\O_R(s^{-2})\\
&={1\over 2s} \RE \< {\rm e}^{i\varphi(s+m(\alpha),GG^*)}\alpha_1,{\rm
  e}^{-i\varphi(s+m(\alpha),G^*G)} \alpha_2\>\\
&\quad+\O_R(s^{-2}).
\end{align*}
A Taylor expansion of $m(v(s))$ implies the existence of a scalar $\widehat m(\alpha)$ such that 
$$m(v(s))-m(\alpha)={\widehat m(\alpha) \over 2s} \RE \< {\rm
  e}^{i\varphi(s+m(\alpha),GG^*)}\alpha_1,{\rm
  e}^{-i\varphi(s+m(\alpha),G^*G)} \alpha_2\>+\O_R(\delta s^{-2}).$$ 
Therefore, there exist functions $\kappa_j(s)$ and $\tilde \kappa_j(s)$, $j\in\{1,2\}$ such that 
$$\|\kappa_j(s)\|_{\mathcal H}+\|\tilde \kappa_j(s)\|_{\mathcal H}+\|\partial_s\kappa_j(s)\|_{\mathcal H}+\|\partial_s\tilde\kappa_j(s)\|_{\mathcal H}=\O_R(\delta)$$
and
\begin{align*}
&\left(m(v(s))-m(\alpha)\right) Jv(s) = \\
&\quad {1\over s} \left( {\rm e}^{3is^2/ 2}\kappa_1(s) +{\rm e}^{-{is^2/2}}
  \tilde \kappa_1(s),  
{\rm e}^{-{3is^2/ 2}}\kappa_2(s) +{\rm e}^{{is^2/ 2}} \tilde
\kappa_2(s)\right)+\O_R(\delta s^{-2}),
\end{align*}
where $J={\rm diag}(1,-1)$ it defined in \eqref{def:JK}.
Define 
\begin{align*}
\kappa^-_1(s) & =-{1\over 2s^2}  \left({\rm e}^{3is^2/ 2}\kappa_1(s) - {\rm e}^{-{is^2/ 2}} \tilde \kappa_1(s)\right),\\
\kappa^-_2(s) & ={1\over 2s^2}  \left({\rm e}^{-3is^2/ 2}\kappa_2(s) - {\rm e}^{{is^2/ 2}} \tilde \kappa_2(s)\right).
\end{align*}
Then the function $\kappa^-(s)=\(\kappa_1^-(s),\kappa_2^-(s)\)$ satisfies
$${1\over i }\partial_s\kappa^-(s) -V\(s+m(\alpha)\)\kappa^-(s)= -\left(m(v(s))-m(\alpha)\right) Jv(s) +\O_R(\delta s^{-2}).$$
On the other hand, we have by definition
\begin{align*}
  \frac{1}{i}\d_s(v-v^-) &= V\(s+m(v(s))\)v -V(s+m(\alpha))v^-\\
&=
  V(s+m(\alpha))(v-v^-) +\(m(v)-m(\alpha)\)Jv.
\end{align*}
The lemma follows by summing the last two identities.
\end{proof}

\appendix

\section{Rigorous derivation for a double-well potential}
\label{sec:justif}

\subsection{Mathematical framework}
\label{sec:framework}

We give more details concerning the derivation of \eqref{eq:NLLZ2} in
the case of a 
condensate in a double well. This is achieved by adapting the approach
from \cite{Sa05}, from which several intermediary results are
borrowed (see also \cite{BaSa07} for some refinements  to \cite{Sa05}
in the confining, one-dimensional case). We shall
derive the model 
\eqref{eq:double-well-env} as an 
envelope equation in the semi-classical limit. Rewrite
\eqref{eq:nlswell} in the presence of physical constants:
\begin{equation}
  \label{eq:double-well-semi}
  i\hbar \frac{\d \psi^\hbar}{\d t} +\frac{\hbar^2}{2}\Delta
  \psi^\hbar = V^\hbar(t,x) 
  \psi^\hbar + \epsilon (\hbar)|\psi^\hbar|^2\psi^\hbar ,
\end{equation}
where $\epsilon (\hbar)$ is a coupling constant whose value will be
discussed later on.
Note that we consider a slightly more general framework than in
Section~\ref{sec:deriv}: $x\in \R^d$, with $d\ge 1$. We assume that
$V^\hbar(t,x)=V_s(x)+\kappa(\hbar)t V_a(x)$, with $V_s$ and $V_a$
independent of $\hbar$. 

\smallbreak

We first describe the assumptions performed on the potential $V_s$ and
discuss the first consequences. 

\begin{hyp}[Symmetric potential]\label{hyp:Vs}
  The potential $V_s\in {\mathcal C}^\infty(\R^d)$ is a smooth real-valued
  function such that: 
  \begin{enumerate}
 \item The potential $V_s$ is at most quadratic,
  \begin{equation*}
    \d^\alpha V_s\in L^\infty(\R^d),\quad \forall |\alpha|\ge 2.
  \end{equation*}
 \item $V_s$ is symmetric with respect to the first coordinate:
    \begin{equation*}
      V_s(-x_1,x_2,\dots,x_d)= V_s(x_1,x_2,\dots,x_d),\quad \forall
      x\in \R^d.
    \end{equation*}
\item $V_s$ admits two minima at $x=x_\pm$, where $x_-$ and $x_+$ are
  distinct and symmetric with respect to the first axis. Moreover,
  \begin{equation*}
    V_s(x)>V_s^{\rm min}= V(x_\pm),\quad \forall x\in \R^d,\ x\not=
    x_{\pm},
  \end{equation*}
and
 \begin{equation*}
V_s^{\rm min} <  \liminf_{|x|\to \infty} V_s(x)=:V_\infty^-.
  \end{equation*}
\item The minima $x_+$ and $x_-$ are non-degenerate critical points: $\nabla
  V(x_\pm)=0$ and $\nabla^2 V(x_\pm)>0$.
  \end{enumerate}
\end{hyp}
\begin{remark}
  As noted in \cite{Sa05}, the last assumption (non-degeneracy) is not
  crucial.
\end{remark}
We denote by
\begin{equation*}
  H_0 = -\frac{\hbar^2}{2}\Delta +V_s.
\end{equation*}
The operator $H_0$  admits a self-adjoint
realization (still denoted by $H_0$) on $L^2(\R^d)$ (see
e.g. \cite{ReedSimon2}).  
Let $\si(H_0)=\si_d\cup \si_{\rm ess}$ be the spectrum of the
self-adjoint operator $H_0$, 
where $\si_d$ denotes the discrete spectrum and $\si_{\rm ess}$
denotes the essential 
spectrum. It follows that 
\begin{equation*}
  \si_d \subset (V_s^{\rm min},V_\infty^-)\quad \text{and}\quad
  \si_{\rm ess}=[V_\infty^-,+\infty).
\end{equation*}
Furthermore, the following two
lemmas hold, which follow from \cite{BeSh91}:  
\begin{lemma}[Lemma~1 from \cite{Sa05}]\label{lem:1}
  For any $\hbar \in (0,\hbar^*]$, for some $\hbar^*$ fixed, it
  follows that: 
  \begin{itemize}
  \item[(i)] $\si_d$ is not empty and, in particular, it contains two
    eigenvalues at least.
\item[(ii)]  The lowest two eigenvalues $\lambda_\pm^\hbar$ of $H_0$
  are non-degenerate, in particular,
  $\lambda_+^\hbar<\lambda_-^\hbar$.  There exists $C>0$, 
independent of $\hbar$, 
such that 
\begin{equation*}
\lambda_\pm^\hbar = V_s^{\rm min}+\O(\hbar);\quad 
  \inf_{\lambda \in \si(H_0)\setminus
    [\lambda_+^\hbar,\lambda_-^\hbar]}\(\lambda-\lambda_\pm^\hbar\)\ge C\hbar.
\end{equation*}
  \end{itemize}
\end{lemma}
\begin{lemma}[\cite{BeSh91}, and Lemma~2 from \cite{Sa05}]\label{lem:2}
  Let $\varphi_\pm^\hbar$ be the normalized eigenvectors associated to
  $\lambda_\pm^\hbar$, then:
  \begin{itemize}
  \item[(i)] $\varphi_\pm^\hbar$ can be chosen to be real-valued functions
    such that
    \begin{equation*}
      \varphi_\pm^\hbar (-x_1,x_2,\dots, x_d) = \pm \varphi_\pm^\hbar (x_1,x_2,\dots, x_d).
    \end{equation*}
\item[(ii)] $\varphi_\pm^\hbar\in \Sigma\cap L^\infty(\R^d)$, where
  \begin{equation*}
    \Sigma =\{f\in H^1(\R^d),\ x\mapsto |x| f(x)\in L^2(\R^d)\}. 
  \end{equation*}
\item[(iii)] There exists $C$ independent of $\hbar$ such that for all
  $\hbar\in (0,\hbar*]$, 
  \begin{align*}
  &  \|\varphi_\pm^\hbar\|_{L^p(\R^d)}\le C
    \hbar^{-\frac{d}{2}\(\frac{1}{2}-\frac{1}{p}\)},\quad \forall p\in
    [2,\infty],\\
& \|\nabla \varphi_\pm^\hbar\|_{L^2(\R^d)}\le C\hbar^{-1/2},\quad \|x
\varphi_\pm^\hbar\|_{L^2(\R^d)}\le C.  
  \end{align*}
  \end{itemize}
\end{lemma}
The single-well states are then defined as in
\eqref{eq:double-well-single}. They satisfy:
\begin{itemize}
\item $\varphi_R^\hbar(-x_1,x_2,\dots,x_d)= \varphi_L^\hbar(x_1,x_2,\dots,x_d)$.
\item $\|\varphi_L^\hbar \varphi_R^\hbar\|_{L^\infty} = \O\( {\rm e}^{-c/\hbar}\)$
  for all $c<\Gamma$, where $\Gamma$ denotes the Agmon distance
  between the two wells
  \begin{equation*}
    \Gamma =\inf_{\gamma \text{ connecting the two wells}} \int_\gamma
    \sqrt{V_s(x)- V_s^{\rm min}}dx>0.
  \end{equation*}
\item For any $r>0$, there exists $C>0$ such that
  \begin{equation*}
    \int_{B(x_+,r)}|\varphi_R^\hbar(x)|^2dx =1 +\O\({\rm e}^{-C/\hbar}\),\quad 
\int_{B(x_-,r)}|\varphi_L^\hbar(x)|^2dx =1 +\O\({\rm e}^{-C/\hbar}\).
  \end{equation*}
\end{itemize}
We denote by
\begin{equation*}
  \Pi_c^\hbar = 1-\( \<\varphi_+^\hbar,\cdot\>\varphi_+^\hbar +
  \<\varphi_-^\hbar,\cdot\>\varphi_-^\hbar\) 
\end{equation*}
the projection onto the eigenspace orthogonal to the bi-dimensional
space associated to $\lambda_\pm^\hbar$. 
Finally, we define the splitting between the lowest two eigenvalues
\begin{equation}\label{eq:splitting}
  \omega^\hbar = \frac{\lambda_-^\hbar-\lambda_+^\hbar}{2}.
\end{equation}
Then for all $c<\Gamma$, $\omega^\hbar =\O\( {\rm e}^{-c/\hbar}\)$.
\smallbreak

We now describe the assumptions made on the potential $V_a$.
\begin{hyp}[Antisymmetric potential]\label{hyp:Va}
 The potential $V_a\in {\mathcal C}^\infty(\R^d)$ is a smooth real-valued
  function such that: 
  \begin{enumerate}
 \item The potential $V_a$ is bounded, as well as its derivatives,
  \begin{equation*}
    \d^\alpha V_a\in L^\infty(\R^d),\quad \forall |\alpha|\ge 0.
  \end{equation*}
 \item $V_a$ is antisymmetric with respect to the first coordinate:
    \begin{equation*}
      V_a(-x_1,x_2,\dots,x_d)= -V_a(x_1,x_2,\dots,x_d),\quad \forall
      x\in \R^d.
    \end{equation*}  
\end{enumerate}
\end{hyp}
\subsection{An approximation result}
\label{sec:approx}

In this section, we prove:
\begin{proposition}\label{prop:approx}
  Let $d\le 2$. Let $V^\hbar$ be as in Section~\ref{sec:framework}, with
  \begin{equation*}
    \kappa=\kappa(\hbar) = \eta \frac{(\omega^\hbar)^2}{\hbar},
  \end{equation*}
for some $\eta\in\R$ independent of $\hbar$.
Suppose that
  $\epsilon(\hbar)$ is given by
  \begin{equation*}
    \epsilon(\hbar) = \delta \omega^\hbar \hbar^{d/2},
  \end{equation*}
where $\delta\ge 0$ does not depend on $\hbar$.  Suppose also that
the initial datum is of the form
\begin{equation*}
  \psi^\hbar(0,x) = \alpha_L\varphi_L^\hbar(x)+\alpha_R\varphi_R^\hbar(x),\quad
  \alpha_L,\alpha_R\in \C\text{ independent of }\hbar. 
\end{equation*}
Define the approximation solution $\psi_{\rm app}$ by
\begin{equation*}
  \psi_{\rm app}^\hbar(t,x) =a_L\(\frac{\omega^\hbar t}{\hbar}\)
  \varphi_L^\hbar(x)+a_R\(\frac{\omega^\hbar t}{\hbar}\)\varphi_R^\hbar(x), 
\end{equation*}
where $(a_L,a_R)=(a_L(\tau),a_R(\tau))$ solves
\begin{equation}\label{eq:double-well-env}
\left\{
  \begin{aligned}
    i\d_\tau a_L& =  \eta \tau  a_L - a_R +
    \delta^\hbar |a_L|^2
    a_L;\quad a_{L\mid \tau=0}=\alpha_L,\\
i\d_\tau a_R& =  -\eta \tau  a_R -  a_L+
    \delta^\hbar|a_R|^2
    a_R;\quad a_{R\mid \tau=0}=\alpha_R,
  \end{aligned}
\right.
\end{equation}
and
\begin{equation*}
  \delta^\hbar = \delta \hbar^{d/2}\int_{\R^d}\varphi_L^4= \delta
  \hbar^{d/2}\int_{\R^d}\varphi_R^4  
\end{equation*}
is uniformly bounded in $\hbar\in (0,\hbar^*]$.
Then we have the following error estimate:
there exist $c,C$ independent of $\hbar$ such that for all $\gamma<\Gamma$,
\begin{equation*}
  \sup_{| t|\le c
    \sqrt \hbar/\omega^\hbar}\|\psi^\hbar(t)-
{\rm e}^{-it(\lambda_-^\hbar+\lambda_+^\hbar)/(2\hbar)}
\psi_{\rm   app}^\hbar(t)\|_{L^2}\le C {\rm e}^{-\gamma/\hbar}. 
\end{equation*}
\end{proposition}
\begin{remark}
  The case $d=1$ and $\delta<0$ could be considered as well, leading
  to the same result. Considering this case would just make the proof
  a bit longer, and we refer to \cite{Sa05} for the adaptation. 
\end{remark}
\begin{remark}
  Since the range for the slow time $\tau =\omega^\hbar t/\hbar$ is
  $-c/\sqrt \hbar \le \tau\le c/\sqrt \hbar$,  the above
  approximation result is a large time result, which is consistent
  with the large time study of \eqref{eq:double-well-env}, or, more
  generally, of \eqref{eq:NLLZ2} to which one reduces thanks to the
  change of variables $s=\sqrt\eta \tau$ (we then have
  $G=-\eta^{-1/2}$ and $\delta F(u)=\eta^{-1/2} \delta^\hbar{\rm diag}
  (|u_1|^2,|u_2|^2)$).  When $\hbar$ is small, Theorem~\ref{theo:scat}
  gives asymptotics for the profiles $a_L$ and $a_R$ of
  $\psi^\hbar_{\rm app}$ for large times  ($|t|\le c{\sqrt\hbar/
    \omega^\hbar}$) and the connexion between the profiles for $t<0$
  and $t>0$ involves the Landau-Zener transition coefficient ${\rm
    e}^{-{\pi\over\eta^2}}$ (at leading order when $\hbar$ goes
  to~$0$).  
\end{remark}

\begin{proof} For simplicity, we write the proof for  $t\ge 0$ and it naturally extends to $t\le0$. \\
{\bf Step 1: Preliminaries.} 
We begin by proving estimates on $\psi^\hbar$, then we perform the rescaling suggested by the form of the approximation solution.
  First, it follows from Lemma~\ref{lem:2} and \cite{Ca11} that for
  fixed $\hbar>0$,  \eqref{eq:double-well-semi} has a unique, global
  solution $\psi^\hbar\in C(\R;\Sigma)$ ($\Sigma$ is defined in
  Lemma~\ref{lem:2}). In addition, if we set
  \begin{align*}
   \text{Mass: } &M^\hbar(t) = \|\psi^\hbar(t)\|_{L^2}^2,\\
\text{Energy: }& E^\hbar(t) = \frac{1}{2}\|\hbar \nabla \psi^\hbar(t)\|_{L^2}^2 +
\frac{\epsilon(\hbar)}{2}\|\psi^\hbar(t)\|_{L^4}^4 + \int_{\R^d}V(t,x)
|\psi^\hbar(t,x)|^2dx ,
  \end{align*}
then we  have the \emph{a priori} estimates:
  \begin{align*}
    &\frac{d M^\hbar}{dt} =0,\\
&  \frac{d E^\hbar}{dt} = \int_{\R^d} \d_t V(t,x)
|\psi^\hbar(t,x)|^2dx= \kappa(\hbar)\int_{R^d} V_a(x)|\psi^\hbar(t,x)|^2dx ,
\end{align*}
where we have used \eqref{eq:double-well-V} for the last
equality. Note that \eqref{eq:double-well-semi} is not a Hamiltonian
equation, unlike the one studied in \cite{BaSa07}, where the
Hamiltonian structure is crucial.  The
conservation of mass and the form of the initial data yield
\begin{equation*}
 M^{\hbar}(t)= \|\psi^\hbar(t)\|_{L^2} =  \|\psi^\hbar(0)\|_{L^2} =M^\hbar(0)\le C,
\end{equation*}
for some $C$ independent of $\hbar$.
We have
\begin{align*}
  E^\hbar (0)&= \<\alpha_L\varphi_L^\hbar + \alpha_R\varphi_R^\hbar,\alpha_L H_0
  \varphi_L^\hbar + \alpha_R H_0\varphi_R^\hbar\> +
  \frac{\epsilon(\hbar)}{2}\|\alpha_L\varphi_L^\hbar +
  \alpha_R\varphi_R^\hbar\|_{L^4}^4 .
\end{align*}
Introduce $\Omega^\hbar=(\lambda_-^\hbar+\lambda_+^\hbar)/2$ and notice that $\Omega^\hbar=V_s^{\rm min} +\O(\hbar)$ by Lemma~\ref{lem:1}. Noting
the identities 
\begin{equation}\label{eq:Hvp}
  H_0\varphi_L^\hbar = \Omega^\hbar \varphi_L^\hbar - \omega^\hbar
  \varphi_R^\hbar;\quad 
H_0\varphi_R^\hbar = \Omega^\hbar \varphi_R^\hbar - \omega^\hbar
  \varphi_L^\hbar,
\end{equation}
we infer from Lemma~\ref{lem:1} and Lemma~\ref{lem:2} that
\begin{equation*}
  E^\hbar (0)= V_s^{\rm min} M^\hbar(0) +\O(\hbar). 
\end{equation*}
Therefore, since $V_a$ is bounded,
\begin{align*}
\|\hbar \nabla \psi^\hbar(t)\|_{L^2}^2 & \le 2 E^\hbar(t) -2 V_s^{\rm
  min} M^\hbar(t)+C\kappa(\hbar)tM^\hbar(t) \\ 
 &\le  2 E^\hbar(0) -2 V_s^{\rm
  min} M^\hbar(0)+2\kappa(\hbar)\int_0^t \int_{R^d}
V_a(x)|\psi^\hbar(t,x)|^2dx+C\kappa(\hbar)t\\ 
&\le C\hbar + C \frac{(\omega^\hbar)^2}{\hbar} t. 
\end{align*}
Next, as in \cite{Sa05}, we consider the slow time
\begin{equation*}
  \tau = \frac{\omega^\hbar t}{\hbar}\ge 0,
\end{equation*}
and the new unknown function
\begin{equation*}
  \Psi^\hbar (\tau,x) = \psi^\hbar (t,x) {\rm e}^{i\Omega^\hbar t/\hbar} =
  \psi^\hbar (t,x) {\rm e}^{i\Omega^\hbar \tau/\omega^\hbar} .
\end{equation*}
It solves
\begin{equation}\label{eq:Psi}
  i\d_\tau \Psi^\hbar = \frac{1}{\omega^\hbar}\(H_0-\Omega^\hbar\)
  \Psi^\hbar +\eta \tau V_a \Psi^\hbar +\delta
  \hbar^{d/2}|\Psi^\hbar|^2\Psi^\hbar, 
\end{equation}
and in view of the above estimates,
\begin{equation}\label{eq:aprioriPsi}
  \|\Psi^\hbar(\tau)\|_{L^2} = \|\Psi^\hbar(0)\|_{L^2}=\O(1);\quad
  \|\nabla \Psi^\hbar(\tau)\|_{L^2}^2 \le C\( \frac{1}{\hbar} +
  \frac{\omega^\hbar}{\hbar^2}\tau\). 
\end{equation}
We decompose $\Psi^\hbar$ as 
\begin{equation}\label{CapitalPsi}
  \Psi^\hbar = \varphi^\hbar + \psi_c^\hbar,\quad
  \psi_c^\hbar=\Pi_c^\hbar \Psi^\hbar,
\end{equation}
so we can write $\varphi^\hbar$ as
\begin{equation*}
  \varphi^\hbar(\tau,x) = \ta_L(\tau)\varphi_L^\hbar(x) +
  \ta_R(\tau)\varphi_R^\hbar(x)  ,
\end{equation*}
for some complex-valued coefficients $\ta_L$ and $\ta_R$. Projecting
\eqref{eq:Psi}, and using \eqref{eq:Hvp}, we find:
\begin{align*}
  i\dot {\tt a}^\hbar_L  & = -\ta_R +\eta \tau \int_{\R^d}V_a \Psi^\hbar
  \varphi_L^\hbar + \delta \hbar^{d/2} \int_{\R^d} |\Psi^\hbar|^2
  \Psi^\hbar \varphi_L^\hbar,\\
i\dot {\tt a}^\hbar_R & = -\ta_L +\eta \tau \int_{\R^d}V_a \Psi^\hbar
  \varphi_R^\hbar + \delta \hbar^{d/2} \int_{\R^d} |\Psi^\hbar|^2
  \Psi^\hbar \varphi_R^\hbar,\\
i\d_\tau \psi_c^\hbar& =
\frac{1}{\omega^\hbar}\(H_0-\Omega^\hbar\)\psi_c^\hbar +\eta \tau
\Pi^\hbar_c\(V_a \Psi^\hbar\)+\delta 
\hbar^{d/2}\Pi_c^\hbar\(|\Psi^\hbar|^2   \Psi^\hbar\).
\end{align*}
The proof of Proposition~\ref{prop:approx} consists in showing that
$\ta_{L/R}$ are close to $a_{L/R}$, and that $\psi^\hbar_c$ is
small. This is achieved in two more steps. 
\smallbreak

\noindent {\bf Step 2. \emph{A priori} estimates on $\ta_R$, $\ta_L$ and $\psi^\hbar_c$.} 
By definition, we have $\ta_L=\int_{\R^d}\Psi^\hbar \varphi_L^\hbar$, so
Cauchy--Schwarz inequality yields
\begin{equation*}
\|\varphi^\hbar(\tau)\|^2_{L^2}=  |\ta_L(\tau)|^2+ |\ta_R(\tau)|^2\le
  \(M^\hbar\)^2\(\|\varphi_L^\hbar\|_{L^2}^2 +
  \|\varphi_R^\hbar\|_{L^2}^2\)\le 2\(M^\hbar\)^2. 
\end{equation*}
Decompose the nonlinearity acting on
$\Psi^\hbar$ as
\begin{equation*}
  |\Psi^\hbar|^2 \Psi^\hbar = |\varphi^\hbar|^2\varphi^\hbar + \tR.
\end{equation*}
Using the \emph{a priori} estimates on $\ta$ and Lemma~\ref{lem:2}, we have
\begin{equation*}
  \left\| |\varphi^\hbar|^2\varphi^\hbar\right\|_{L^2}\le
  \|\varphi^\hbar\|_{L^\infty}^2 \|\varphi^\hbar\|_{L^2}\le
  C\(\|\varphi_L^\hbar\|_{L^\infty}^2 +  \|\varphi^\hbar_R\|_{L^\infty}^2\)\le C\hbar^{-d/2}.
\end{equation*}
Since we have the pointwise estimate
\begin{equation*}
 |\tR|\le C\( |\varphi^\hbar|^2 |\psi^\hbar_c| + |\psi^\hbar_c|^3\), 
\end{equation*}
we infer
\begin{equation*}
  \|\tR\|_{L^2}\le C\( \hbar^{-d/2} \|\psi^\hbar_c\|_{L^2} + 
\|\psi^\hbar_c\|^3_{L^6}\),
\end{equation*}
and Gagliardo--Nirenberg inequality yields ($d\le 2$)
\begin{equation*}
  \|\psi^\hbar_c\|^3_{L^6(\R^d)} \le C
  \|\psi^\hbar_c\|_{L^2(\R^d)}^{3-d}
\|\nabla\psi^\hbar_c\|_{L^2(\R^d)}^{d}.
\end{equation*}
Since $d\le 2$, we can factor out $\|\psi^\hbar_c\|_{L^2(\R^d)}$, and
\eqref{eq:aprioriPsi} gives
\begin{equation*}
  \|\psi^\hbar_c\|^3_{L^6(\R^d)} \le C
  \|\psi^\hbar_c\|_{L^2(\R^d)}\hbar^{-d/2}\(1 + \(\frac{\omega^\hbar
    \tau}{\hbar}\)^{d/2}\).
\end{equation*}
Therefore,
\begin{equation}
  \label{eq:aprioriR}
   \|\tR\|_{L^2}\le C \hbar^{-d/2}\(1 + \(\frac{\omega^\hbar
    \tau}{\hbar}\)^{d/2}\)\|\psi^\hbar_c\|_{L^2},
\end{equation}
and $|\Psi^\hbar|^2 \Psi^\hbar$ satisfies a similar estimate. We infer 
\begin{equation*}
  \left\lvert\dot {\tt a}^\hbar_L(\tau)\right\rvert + \left\lvert\dot {\tt
      a}^\hbar_R(\tau)\right\rvert \le C \(1+\tau\) +C\(1 +
  \(\frac{\omega^\hbar    \tau}{\hbar}\)^{d/2}\) .
\end{equation*}
Since $d\le 2$ and $\omega^\hbar   =\O({\rm e}^{-c/\hbar})$, we can simplify
the above estimate:
\begin{equation*}
  \left\lvert\dot {\tt a}^\hbar_L(\tau)\right\rvert + \left\lvert\dot {\tt
      a}^\hbar_R(\tau)\right\rvert \le C \(1+\tau\),\quad \forall
  \tau\ge 0. 
\end{equation*}
From this we infer
\begin{equation}\label{eq:aprioridtauphi}
  \left\lVert \d_\tau \varphi^\hbar\right\rVert_{L^2} \le
  \left\lvert\dot {\tt a}^\hbar_L(\tau)\right\rvert
\|\varphi_L^\hbar\|_{L^2} + \left\lvert\dot {\tt a}^\hbar_R(\tau)\right\rvert
\|\varphi_R^\hbar\|_{L^2} \le C(1+\tau),
\end{equation}
and, again from Lemma~\ref{lem:2},
\begin{equation}\label{eq:aprioridtauphi3}
   \left\lVert \d_\tau \(\lvert
    \varphi^\hbar\rvert^2\varphi^\hbar\)\right\rVert_{L^2}  \le
  3\|\varphi^\hbar\|_{L^\infty}^2\left\lVert \d_\tau
    \varphi^\hbar\right\rVert_{L^2}   \le C\hbar^{-d/2}(1+\tau). 
\end{equation}
\smallbreak

\noindent {\bf Step 3. Stability estimates.} In view of~\eqref{CapitalPsi}, we first prove that $\psi^\hbar_c$
is small. Since $\psi^\hbar_{c\mid \tau=0}=0$, Duhamel's formula
yields
\begin{align}\label{align}
  \psi^\hbar_c(\tau) &= -i\eta\int_0^\tau
  {\rm e}^{-i(H_0-\Omega^\hbar)(\tau-s)/\omega^\hbar} \(s \Pi^\hbar_c\( V_a
  \Psi^\hbar\)(s)\)ds\\
  \nonumber
&\quad -i \delta \hbar^{d/2}\int_0^\tau
  {\rm e}^{-i(H_0-\Omega^\hbar)(\tau-s)/\omega^\hbar} \Pi^\hbar_c\( |\Psi^\hbar|^2
  \Psi^\hbar\)(s)ds.
\end{align}
Each term is treated in a similar fashion: when $\Psi^\hbar$ is
replaced by $\varphi^\hbar$, we perform an integration by
parts, and for the remaining term,  we use directly the \emph{a priori}
estimates. For the first part of~\eqref{align}, we write
\begin{equation*}
  \Pi^\hbar_c\( V_a
  \Psi^\hbar\)= \Pi^\hbar_c\( V_a
  \varphi^\hbar\) + \Pi^\hbar_c\( V_a
  \psi^\hbar_c\),
\end{equation*}
and set
\begin{align*}
  I^\hbar_1(\tau) & = \int_0^\tau
  {\rm e}^{-i(H_0-\Omega^\hbar)(\tau-s)/\omega^\hbar} \(s \Pi^\hbar_c\( V_a
  \varphi^\hbar\)(s)\)ds,\\
I^\hbar_2(\tau) & = \int_0^\tau
  {\rm e}^{-i(H_0-\Omega^\hbar)(\tau-s)/\omega^\hbar} \(s \Pi^\hbar_c\( V_a
  \psi^\hbar_c\)(s)\)ds.
\end{align*}
Integrating by parts,
\begin{align*}
  I^\hbar_1(\tau) & = -i\omega^\hbar 
  {\rm e}^{-i(H_0-\Omega^\hbar)(\tau-s)/\omega^\hbar}
  \(H_0-\Omega^\hbar\)^{-1}\Pi^\hbar_c\(s  V_a 
  \varphi^\hbar\)(s)\Big|_0^\tau\\
+i\omega^\hbar &\int_0^\tau
  {\rm e}^{-i(H_0-\Omega^\hbar)(\tau-s)/\omega^\hbar}
  \(H_0-\Omega^\hbar\)^{-1}\(\Pi^\hbar_c\( V_a 
  \varphi^\hbar\)+ s \Pi^\hbar_c\( V_a
  \d_\tau \varphi^\hbar\)\)(s)ds .
\end{align*}
From Lemma~\ref{lem:1}, there exists $C$ independent of $\hbar\in
(0,\hbar^*]$ such that
\begin{equation*}
  \left\lVert \hbar
    \(H_0-\Omega^\hbar\)^{-1}\Pi^\hbar_c\right\rVert_{L^2\to L^2}\le
  C. 
\end{equation*}
We infer, using \eqref{eq:aprioridtauphi},
\begin{equation*}
  |I^\hbar_1(\tau) |\le 
  C\frac{\omega^\hbar }{\hbar} \(\tau + \tau^3\). 
\end{equation*}
For $I^\hbar_2$, we have directly
\begin{equation*}
  |I^\hbar_2(\tau) |\le 
  C\int_0^\tau s\|\psi_c^\hbar(s)\|_{L^2}ds.
\end{equation*}
For the nonlinear term in Duhamel's formula (the second term of~\eqref{align}), we also write 
$$\Pi^\hbar_c(|\Psi^\hbar|^2\Psi^\hbar)= \Pi^\hbar_c(|\varphi^\hbar|^2\varphi^\hbar)+\Pi^\hbar_c\tR,$$
and set
\begin{align*}
  I^\hbar_3(\tau) & =\hbar^{d/2} \int_0^\tau
  {\rm e}^{-i(H_0-\Omega^\hbar)(\tau-s)/\omega^\hbar}\Pi^\hbar_c\( |\varphi^\hbar|^2
  \varphi^\hbar\)(s)ds,\\
I^\hbar_4(\tau) & = \hbar^{d/2}\int_0^\tau
  {\rm e}^{-i(H_0-\Omega^\hbar)(\tau-s)/\omega^\hbar}  \Pi^\hbar_c \tR (s)ds.
\end{align*}
We have, directly from \eqref{eq:aprioriR},
\begin{equation*}
  |I^\hbar_4(\tau) |\le 
  C\int_0^\tau  \(1 + \(\frac{\omega^\hbar
    s}{\hbar}\)^{d/2}\)\|\psi_c^\hbar(s)\|_{L^2}ds,
\end{equation*}
and performing an integration by parts for $I_3^\hbar$, using
\eqref{eq:aprioridtauphi3}, we have
\begin{equation*}
  |I^\hbar_3(\tau) |\le 
  C\frac{\omega^\hbar}{\hbar} \(1+\tau^2\). 
\end{equation*}
Since $d\le 2$ and $\omega^\hbar$ decays exponentially, we come up with:
\begin{equation*}
  \|\psi_c^\hbar(\tau)\|_{L^2}\le C\( \frac{\omega^\hbar}{\hbar}
  \(1+\tau^3\) + \int_0^\tau (1+s)\|\psi_c^\hbar(s)\|_{L^2}ds\).
\end{equation*}
Gronwall lemma yields
\begin{equation*}
  \|\psi_c^\hbar(\tau)\|_{L^2} \le C
  \frac{\omega^\hbar}{\hbar}(1+\tau^3){\rm e}^{C(\tau+\tau^2)}. 
\end{equation*}
Recalling again that $\omega^\hbar$ decays exponentially, we can write
that for all $c_0<\Gamma$ (the Agmon distance between the two wells),
\begin{equation*}
   \|\psi_c^\hbar(\tau)\|_{L^2} \le C
  (1+\tau^3){\rm e}^{C(\tau+\tau^2)-c_0/\hbar}. 
\end{equation*}
This is small as $\hbar\to 0$, provided that $\tau^2\ll/\hbar$. More
precisely, there exist  $c_1,c_2>0$ independent of $\hbar$ such that
\begin{equation}\label{eq:estpsic}
  \|\psi_c^\hbar(\tau)\|_{L^2} \le C {\rm e}^{-c_1/\hbar},\quad 0\le \tau\le
  \frac{c_2}{\sqrt\hbar}. 
\end{equation}
To conclude the proof of the proposition, set
\begin{equation*}
  \tw_L=\ta_L-a_L;\quad \tw_R= \ta_R-a_R. 
\end{equation*}
Subtracting the equation for $a_L$ from the equation for $\ta_L$, we
have
\begin{align*}
  i\dot {\tt w}_L^\hbar &= -\tw_R +\eta \tau \int_{\R^d}V_a
  \(\Psi^\hbar - a_L\varphi^\hbar_L\)\varphi_L^\hbar\\
&\quad +\delta \hbar^{d/2}\int_{\R^d}\(|\Psi^\hbar|^2\Psi^\hbar
-|a_L|^2a_L |\varphi_L^\hbar|^2 \varphi_L^\hbar \)\varphi_L^\hbar. 
\end{align*}
We have
\begin{equation*}
  \int_{\R^d}V_a
  \(\Psi^\hbar - a_L\varphi^\hbar_L\)\varphi_L^\hbar= \int_{\R^d}V_a
  \(\ta_R \varphi^\hbar_R + \psi_c^\hbar +\tw_L
  \varphi^\hbar_L\)\varphi_L^\hbar ,
\end{equation*}
therefore, since the product $\varphi_L^\hbar\varphi_R^\hbar$ decays
exponentially in $\hbar$, 
\begin{equation*}
  \left| \int_{\R^d}V_a
  \(\Psi^\hbar - a_L\varphi^\hbar_L\)\varphi_L^\hbar\right|\le
C\({\rm e}^{-c/\hbar} + \|\psi^\hbar_c\|_{L^2} + |\tw_L|\).
\end{equation*}
A similar estimate can be established for the other source term in the
equation for $\tw_L$. The equation for $\tw_R$ is handled in the same
fashion, and using \eqref{eq:estpsic}, we end up with:
\begin{equation*}
  |\dot {\tt w}_L^\hbar(\tau)| + |\dot {\tt w}_R^\hbar(\tau)|\le C\(
  |\tw_L(\tau)| 
  +|\tw_R(\tau)| + {\rm e}^{-c/\hbar} \), \quad 0\le \tau \le \frac{c_2}{\sqrt\hbar}.
\end{equation*}
Gronwall lemma and \eqref{eq:estpsic} then yield 
Proposition~\ref{prop:approx}. 
\end{proof}

\end{document}